\documentclass[12pt]{article}
\usepackage{fullpage}
\usepackage{amsmath}  %for \align
\usepackage{amssymb, amsthm, algorithm, algorithmic,thmtools,thm-restate,enumerate,verbatim,bm}
\usepackage{mathrsfs}
\usepackage[usenames,dvipsnames,svgnames,table]{xcolor}
\usepackage{tikz}
\definecolor{darkgreen}{rgb}{0.0,0,0.9}
\usepackage[colorlinks=true,pdfpagemode=none,citecolor=OliveGreen,linkcolor=BrickRed,urlcolor=BrickRed,pdfstartview=FitH]{hyperref}
\usepackage[percent]{overpic}
\usepackage{sidecap}
\usepackage[margin=10pt,font=small,labelfont=bf]{caption}
\usepackage{microtype}

\renewcommand{\P}[1]{{\mathbb{P}}\left[#1\right]}

\newcommand{\PP}[2]{{\mathbb{P}}_{#1}\left[#2\right]}

\newcommand{\I}[1]{{\mathbb{ I}}\left[#1\right]}
\DeclareMathOperator{\E}{\mathbb{E}}

\providecommand{\norm}[1]{\left\lVert#1\right\rVert}
\providecommand{\wnorm}[1]{\left\lVert#1\right\rVert_\weight}

\numberwithin{equation}{section}
\declaretheorem[numberwithin=section]{theorem}
\declaretheorem[sibling=theorem]{lemma}
\declaretheorem[sibling=theorem]{proposition}

\declaretheorem[sibling=theorem]{corollary}

\declaretheorem[sibling=theorem]{fact}
\declaretheorem[style=remark,sibling=theorem]{example}

\newenvironment{proofof}[1]{{\medbreak\noindent \em Proof of #1.  }}{\hfill\qed\medbreak}
\newenvironment{proofofnoqed}[1]{{\medbreak\noindent \em Proof of #1.  }}{\medbreak}

\def\reff{{\mathcal R}_{\textup{eff}}}
\def\ceff{{\mathcal C}_{\textup{eff}}}
\def\rdiam{{\mathcal R}_{\textup{diam}}}

\def\b#1{{\bf #1}}
\def\eps{{\epsilon}}
\def\cD{{\cal D}}
\def\cL{{\cal L}}
\def\Iz{I^*}
\def\bv{{\bf v}}

\def\be{{\bf e}}
\def\muz{\mu^*}
\def\cE{{\cal E}}
\def\cL{{\cal L}}
\def\cH{{\cal H}}
\def\cB{{\mathscr B}}
\def\cR{\mathsf{Ray}}  % Rayleigh quotient
\def\cP{\mathcal{P}}
\def\weight{w}

\def\spheresym{centered}

\DeclareMathOperator{\sspan}{span}
\DeclareMathOperator{\img}{img}

\DeclareMathOperator{\dist}{{dist}}
\DeclareMathOperator{\nei}{nbd}
\DeclareMathOperator{\poly}{poly}
\DeclareMathOperator{\diam}{diam}

\DeclareMathOperator{\vol}{\mathsf{wt}}
\def\dfn#1{\textit{\textbf{#1}}}
\def\ceil#1{\lceil #1 \rceil}
\def\flr#1{\lfloor #1 \rfloor}
\def\iprod#1{\langle #1 \rangle}
\def\wiprod#1{\langle #1 \rangle_\weight}
\def\Z{{\mathbb Z}}
\def\R{{\mathbb R}}
\def\C{{\mathbb C}}
\def\N{{\mathbb N}}
\def\FT{{\mathcal F}}
\DeclareMathOperator{\Ent}{Ent}
\def\tcommute{t_\leftrightsquigarrow}
\def\gh{G}
\def\edges{E}
\def\verts{V}
\def\vertex{V}
\def\cp{{\boldsymbol\tau}}
\DeclareMathOperator{\tr}{tr}
\DeclareMathOperator{\detp}{det^\prime}
\def\cbuldot{{\raise.25ex\hbox{$\scriptscriptstyle\bullet$}}}
\begin{document}

\title{\bf Sharp Bounds on Random Walk Eigenvalues\\ via Spectral Embedding}

\author{
Russell Lyons\thanks{Department of Mathematics, Indiana
University. Partially supported by the National
Science Foundation under grants DMS-1007244 and DMS-1612363
and by Microsoft Research.
Email: \protect\url{rdlyons@indiana.edu}.}
\and
Shayan Oveis Gharan\thanks{Department of Computer Science and Engineering,
University of Washington. Supported by a Stanford Graduate Fellowship. 
Part of this work was done while the author was a summer intern at
Microsoft Research, Redmond. Email: \protect\url{shayan@cs.washington.edu}.}
}

\maketitle
\begin{abstract}
Spectral embedding of graphs uses the top $k$ non-trivial
eigenvectors of the random
walk matrix to embed the graph into $\mathbb{R}^k$. The primary use of
this embedding has been
for practical spectral clustering algorithms \cite{SM00,NJW02}.
Recently, spectral embedding was studied from a
theoretical perspective to prove higher order variants of Cheeger's
inequality \cite{LOT12,LRTV}.

We use  spectral embedding to provide a unifying framework for bounding all
the eigenvalues of  graphs.
 For example, we show that for any finite connected graph with $n$ vertices and all
 $k \ge 2$,
 the $k$th largest eigenvalue is at most
 $1-\Omega(k^3/n^3)$, which extends the only other such result known, which
 is for $k=2$ only and is due to
 \cite{LO81}. This upper bound improves to $1-\Omega(k^2/n^2)$ if
 the graph is regular. We generalize these results, and we provide sharp
 bounds on the spectral measure of various classes of graphs, including vertex-transitive graphs and infinite graphs, in terms of specific graph parameters like the volume growth. 
 
 As a consequence, using the entire spectrum, we provide (improved) upper
 bounds on the return probabilities and mixing time of random walks with
 considerably shorter and more direct proofs. Our work introduces spectral embedding as a new tool in analyzing reversible Markov chains. 
Furthermore, building on  \cite{Lyo05}, we design a
local algorithm to approximate the number of spanning trees  of massive graphs.
\end{abstract}

\newpage
\tableofcontents
\section{Introduction}
A very popular technique for clustering data involves forming a (weighted)
graph whose vertices are the data points and where the weights of the edges
represent the ``similarity'' of the data points. Several of the
eigenvectors of one of the Laplacian matrices of this graph are then used
to embed the graph
into a moderate-dimensional Euclidean space. Finally,  one partitions
the vertices using $k$-means or other heuristics. This is known as spectral
embedding or spectral clustering, and it is applied in various practical
domains  (see, e.g., \cite{SM00,NJW02,Luxburg07}). 
Recently, theoretical justifications of some of these algorithms have been
given. For example, \cite{LOT12, LRTV, DJM} used 
spectral embedding  to prove higher order variants of Cheeger's inequality,
namely, that a graph can be partitioned into $k$ subsets each defining a sparse cut if and only if   the $k$th smallest eigenvalue of the normalized Laplacian is close to zero.

Spectral embedding for finite graphs is easy to describe. For simplicity in
this paragraph,
let $G=(V,E)$ be a $d$-regular, connected graph, and let $A$ be the adjacency matrix of  $G$. Then the normalized Laplacian of $G$ is $\cL := I - A/d$. 
 Let $g_1,\ldots,g_k$ be orthonormal eigenfunctions of $\cL$ 
corresponding to the $k$ smallest eigenvalues $0 = \lambda_1 < \lambda_2
\le \cdots \le \lambda_k$.
Then (up to normalization) the spectral embedding is the function $F\colon
V\rightarrow \mathbb{R}^{k-1}$ defined by
$$ x \mapsto F_x:=\bigl(g_2(x),g_3(x),\ldots,g_k(x)\big)\,.$$
This embedding satisfies  interesting properties, including one termed
``isotropy"  by \cite[Lemma 3.2]{LOT12}. This isotropy property says that
for any unit vector $\bv\in\mathbb{R}^{k-1}$,
$$ \sum_{x\in V} d \langle \bv,F_x\rangle^2=1\,.$$
The embedding is naturally related to the
eigenvalues of $\cL$. For example, it is straightforward that 
\begin{equation}
\label{eq:Fenergylambdak}
(k-1) \lambda_k \geq \sum_{x\sim y} \norm{F_x-F_y}^2\,.
\end{equation}
Let the {\em energy} of $F$ be the
value of the right-hand side of the above inequality.
It follows from the variational principle that the spectral embedding is
an embedding that {\em minimizes} the energy among all isotropic embeddings.
(Note that the embedding that only minimizes the energy is the one that maps
every vertex to the same point in $\mathbb{R}^{k-1}$.)

In fact, we will not use the isotropy property explicitly, except in
\autoref{lem:isotropy}. The reason is
that rather than use the above version of the spectral embedding, we use an
isomorphic one that is defined via a spectral projection. The fact that a
projection is behind the definition of the embedding is what makes isotropy
hold.

In this paper, we use spectral embedding as a unifying framework to
bound from below all the eigenvalues of the 
normalized Laplacian of (weighted) graphs. 
We prove universal lower bounds on these eigenvalues, equivalently,
universal upper bounds on the eigenvalues of the random walk
matrix of $G$.
The usual methods for obtaining such bounds involve indirect methods from
functional analysis.
By contrast, our method is direct, which leads to very short proofs, as
well as to improved bounds. By \eqref{eq:Fenergylambdak}, all we need to do 
is to bound from below the energy of an isotropic embedding.
We use simple properties of Hilbert spaces, as well as underlying properties of
$G$, to achieve this goal.

There have been a great many papers that upper-bound  the
return probability or the mixing time of random walks. 
It is known that return probabilities are closely related to the vertex
spectral measure (see Lemmas \ref{lem:returnprobspectral} and
\ref{lem:reversereturnprobspectral} for specific comparisons).
Therefore, once we can control the eigenvalues, we can reproduce, or even
improve,
bounds on return probabilities. Our work thus
introduces spectral embedding as a new tool in analyzing reversible Markov
chains.

\subsection{Results}
In order to give an overview of our results, we need the following
notation, which is explained in more detail in Sections \ref{sec:prelim}
and \ref{sec:spectrals}.
To simplify, we consider only unweighted simple connected
graphs in this introduction.
Consider lazy simple random walk, which stays put with probability
1/2 and moves to a random uniform neighbor otherwise.
Denote the transition matrix of this random walk by $P$.
The probabilistic Laplacian matrix is $\cL := I - P$.
If $G$ is finite of size $n$, then the eigenvalues of $\cL$ are $0 =
\lambda_1 < \lambda_2 \le \cdots \le \lambda_n \le 1$.
If $G$ is infinite, there may not be any eigenvectors in $\ell^2(V)$, so one
defines instead a spectral probability measure $\mu_x$ on $[0, 1]$
corresponding to each vertex $x \in V$.
One way to define $\mu_x$ is via random walks.
Write $p_t(x, x)$ for the probability that random
walk started at $x$ is back at $x$ on the $t$th step.
Then
$$
p_t(x,x) =
\int_0^1(1-\lambda)^{t} \,d\mu_x(\lambda)\,.
$$
In the finite case, we define $\mu := \sum_{x \in V} \mu_x/n$, where $n :=
|V|$.
In this case, $\mu(\delta) := \mu\bigl([0, \delta]\big)
= \max \{k/n : \lambda_k \le \delta\}$, in other words, $\mu$ places mass
$1/n$ at each of the $n$ eigenvalues of $\cL$ (with multiplicity).
Write $\pi(x)$ for the degree of $x$ divided by $2|E|$, which is 0 when $G$
is infinite.
It will be more convenient to use $\muz_x := \mu_x - \pi(x) \b1_0$ and
$\muz := \mu - \b1_0/n$, where $\b1_0$ denotes the point mass at 0.
If $G$ is vertex transitive, then $\mu_x$ does not depend on $x$, so we
write $\mu := \mu_x$ and $\muz := \muz_x$, which agrees with our notation
in the preceding sentence in case $G$ is finite.

Our main contributions are the following results, all of which we believe
to be new, as well as the technique used to establish them.
The sharpness of these results (up to a constant factor) is discussed
briefly here and in more detail in the body of the paper. 

\begin{theorem}
\label{thm:generalgraphseigintro}
For every finite, unweighted, connected graph $G$, and every
$\delta\in(0,1)$, we have $\muz(\delta) <  20\,\delta^{1/3}$ and
$$\lambda_k > \frac{(k-1)^3}{(20\,n)^3}\,.$$
Thus, for every integer $t\geq 1$, we have
$$
\frac{\sum_{x\in V} p_t(x,x) - 1}{n} < \frac{13}{t^{1/3}}
\,.
$$
\end{theorem}

Here, the first result is sharp for each $k$ separately and the second
result is sharp.
(Note that when this theorem is stated and proved as
\autoref{thm:generalgraphseig}, it is for the Laplacian corresponding to
the transition matrix for simple random walk, rather than for lazy random
walk as here. This is why the constants differ.)

Our main application of the above result is a fast local algorithm for approximating the number $\cp(G)$ of
spanning trees of a finite massive graph, $G$.
The problem of counting the number of spanning trees of a graph is one of
the fundamental problems in graph theory, for which the matrix-tree theorem
gives a simple $O(n^3)$-time algorithm.
For very large $n$, however, even this is too slow.
For a general graph, $\cp(G)$ can be as large as $n^{n-2}$, which is its
value for a complete graph by Cayley's theorem \cite{cayley89}.

A {\em local} graph algorithm is one that is allowed to look  only at the
local neighborhood of  random samples of vertices of the graph.
The notion of graph-parameter estimability involves estimating a graph
parameter, such as $\cp(G)$,  using a local graph algorithm (see, e.g.,
\cite{Elek:test} or \cite[Chap.~22]{Lovasz:graphlimits} for a discussion). 
We prove that $\cp(G)$ is estimable in this sense.
In fact, we prove estimability in an even stronger sense.
Suppose that we have access to $G$ only through an oracle that supports the following simple operations:
\begin{itemize}
\item  Select  a uniformly random vertex of $G$.
\item For a given vertex $x\in V$,  select a uniformly random neighbor of $x$.
\item For a given vertex $x\in V$, return the degree of $x$.
\end{itemize} 
The proof of the next corollary presents
a local algorithm for approximating the
number of spanning trees of $G$ that uses an oracle
satisfying the above operations, as well as knowledge of $n$ and $|E|$.
For any given $\eps>0$, our algorithm approximates
$\frac{1}{n}\log\cp(G)$ within an $\eps$-additive error using only
$O\bigl(\poly(\epsilon^{-1}\log n)\big)$ queries. 

\begin{corollary}
Let $G$ be a finite, unweighted, connected graph. Given an oracle access to
$G$ that satisfies the above operations, together
with knowledge of $|V|$ and $|E|$,
there is a randomized algorithm that for any given $\eps,\delta>0$,
approximates $\log\cp(G) / |V|$ within an additive error of $\eps$, with
probability at least $1-\delta$, by using only
$\tilde{O}\bigl(\eps^{-5}+\eps^{-2}\log^2{|V|}\big)\log \delta^{-1}$ many oracle
queries. 
\end{corollary}

Here, we write $f(s) = \tilde{O}\bigl(g(s)\big)$ if there is a constant $c$
such that $f(s) \le c\, g(s) \bigl(\log g(s)\big)^c$ for all $s$.

Write $f(n) = \Omega\bigl(g(n)\bigr)$ to mean that there is a positive
constant $c$ such that $f(n) \ge c\, g(n)$ for all $n \ge 1$. In this
notation,
the preceding \autoref{thm:generalgraphseigintro} gave an
$\Omega\bigl((k-1)^3/n^3\big)$ bound for $\lambda_k$.
With the additional hypothesis of regularity, this
can be improved to $\Omega\bigl((k-1)^2/n^2\big)$. In fact, only a bound for
the ratio of the maximum degree to the minimum degree is needed.
\begin{theorem}
\label{prop:regularreturnintro}
For every unweighted, connected, regular graph $G$ and every $x\in V$, we
have $\muz_x(\delta) < 14\sqrt{\delta}$. 
Hence if $G$ is finite, $\muz(\delta) < 14\sqrt{\delta}$ and
for $1\leq k\leq n$, we have
$$\lambda_k > \frac{(k-1)^2}{200\,n^2}\,.$$
For all $t > 0$ and $x \in V$, we have 
$$
p_t(x, x) - \pi(x)
<
\frac{13}{\sqrt t}
\,.
$$
\end{theorem}

This result is evidently sharp as shown by the example of a cycle, which
also shows sharpness of the next result.

For a finite $G$,
let $\tau_\infty(1/4)$ denote the uniform mixing time, i.e., the time $t$
until $|p_t(x, y)/\pi(y) - 1| \le 1/4$ for every $x, y \in V$.

\begin{proposition}
\label{cor:regularmixintro}
For every unweighted, finite, connected regular graph $G$, we have
$$ 
\tau_\infty(1/4) \leq 24\, n^2
\,.
$$
\end{proposition}

The next theorem answers (up to constant factors) the 5th open question in
\cite{MT06}, which asks how small the log-Sobolev and entropy constants can
be for an $n$-vertex unweighted connected graph.
Here, we write $f(n) = \Theta\bigl(g(n)\bigr)$ to mean that there are
positive finite constants $c_1$ and $c_2$ such that for all $n \ge 1$, we
have $c_1 g(n) \le f(n) \le c_2 g(n)$.

\begin{theorem}
\label{thm:SobEntConstsintro}
Write $\rho(G)$ for the log-Sobolev constant and $\rho_0(G)$ for the entropy
constant of $G$.
For finite unweighted graphs $G$ with $n$ vertices, we have
\begin{equation*}
\min_G \rho(G) = \Theta(n^{-3})
\quad \mbox{and} \quad
\min_G \rho_0(G) = \Theta(n^{-3})
\,.
\end{equation*}
\end{theorem}

Similarly, we find the worst uniform mixing time of graphs:
\begin{proposition}
\label{cor:inftymixboundintro}
For every unweighted, finite, connected graph $G$, we have
$$ \tau_\infty(1/4)\leq 8n^3\,.$$
\end{proposition}

This result is sharp.

Although new, the preceding three results have been known implicitly in the
sense that they could have been easily deduced from known results, but for
some reason, they were not.

Finally, the case of transitive graphs is especially interesting and
especially well studied,
yet, to the best of our knowledge, the following theorem has not been proved in this generality.

\begin{theorem}
\label{thm:transitiveeigintro}
For every connected, unweighted, vertex-transitive, locally finite graph
$G$ of degree $d$,
every $\alpha\in (0,1)$, $\delta\in (0, 1)$, and every $x\in V$, 
\begin{equation}
\label{eq:fullbound}
\muz_x(\delta) = \muz(\delta) \leq \frac{1}{(1-\alpha)^2
N\bigl(\!\sqrt{\alpha/(d\delta)}\,\big)}\,,
\end{equation}
where $N(r)$ denotes the number of vertices in a ball of radius $r$.
In addition, if $G$ is finite of diameter $\diam$, then 
\begin{equation}
\label{eq:improvedl2intro}
\lambda_2
>
\frac2d
\left(\sin \frac{\pi}{4 \diam} \right)^2
\,.
\end{equation}
\end{theorem}

Because the volume-growth function $N(\cdot)$ can fluctuate dramatically in
some groups, it can be important to have $N(\cdot)$ appear more directly in
the estimate as in \eqref{eq:fullbound}. For example, see \cite{LPS:WSF}
for an application to occupation measure of random walks in balls where
this bound is crucial.

The first lower bound for $\lambda_2$ on finite Cayley graphs similar to
\eqref{eq:improvedl2intro}
is due to \cite[Lemma 6.1]{Babai91}; the constant was improved later by
\cite[Corollary 1]{DSC:compare} to
\begin{equation}
\label{eq:transitivegapintro}
\lambda_2 > \frac{1}{d\diam^2}
\,.
\end{equation}
It is known that the same inequality holds for general
finite transitive graphs; it can be proved by the congestion method, e.g.,
\cite[Corollary 13.24]{LPW06}.

Note that \eqref{eq:improvedl2intro} agrees with
\eqref{eq:transitivegapintro} when $\diam = 1$ and is otherwise strictly better.
Also,
\[
\frac2d
\left(\sin \frac{\pi}{4 \diam} \right)^2
\sim \frac{\pi^2}{8 d \diam^2}
\]
as $\diam \to \infty$.
Our improvement is accomplished through adapting a proof 
due to \cite{JL}
of similar result for compact manifolds, 
due originally to \cite{Li}.
(\cite{JL} in turn was inspired by an earlier version of the present
paper.)

Our technique yields very short proofs of the above results.
In addition, one can immediately
deduce such results as
that return probabilities in infinite transitive graphs with
polynomial growth at least order $D$ decay at polynomial rate at least
order $D/2$ (see \autoref{cor:polygrowthtrans}).
This is, of course, the correct decay rate on $\Z^D$ for $D \in \N$.

\subsection{Related Works}
There have been many studies bounding from above
the eigenvalues of the (normalized) Laplacian (equivalently, bounding the
eigenvalues of the (normalized) adjacency matrix from below). For example,
Kelner et al.~\cite{KLPT11} show that  for $n$-vertex, bounded-degree
planar graphs, one has that the $k$th smallest eigenvalue satisfies
$\lambda_k = O(k/n)$. 

However, to the best of our knowledge, universal lower bounds were known only for the second
smallest eigenvalue of the normalized Laplacian. Namely, Landau and Odlyzko
\cite{LO81} showed that the second eigenvalue of every simple connected graph
of size $n$ is at least $1/n^3$.

On the other hand, there have also been a great many papers that bound from
above the return probabilities of random walks, both on finite and infinite
graphs. 
Such bounds correspond to lower bounds on eigenvalues.
In fact, as we review in
\autoref{subsec:ran}, the asymptotics of large-time return probabilities
correspond to the asymptotics of the spectral measure near 0, which, for
finite graphs, means the behavior of the smallest eigenvalues.

Our methods would work as well for the eigenvalues $\widetilde\lambda_k$
of the unnormalized combinatorial Laplacian $\Delta$.
This is relevant for continuous-time random walk that when at a vertex $x$,
crosses each edge $(x, y)$ at
rate equal to the weight $w(x, y)$ of that edge.
\vadjust{\kern2pt}%
In this case, \cite{Friedman} has determined the minimum of
$\widetilde\lambda_k$ for each $k$ over all unweighted $n$-vertex graphs.
\vadjust{\kern2pt}%
As noted there, his bound implies that $\widetilde\lambda_k =
\Omega(k^2/n^2)$; this immediately implies that $\lambda_k =
\Omega(k^2/n^3)$ by comparison of Rayleigh quotients, but this is not
sharp, as indicated by \autoref{thm:generalgraphseigintro}.

\subsection{Structure of the Paper}
We review background and notation for graphs in \autoref{sec:prelim}
and for spectral embedding and random walks in
\autoref{sec:spectrals}. We then begin with
some very simple proofs of known results in \autoref{sec:lower}.
Those are
followed by proofs of new results that lead to the above bounds on mixing
time, log-Sobolev constants, and entropy constants.
Our most sophisticated proof is in \autoref{sec:loweravg}, which establishes
\autoref{thm:generalgraphseigintro} and its corollaries.
The case of transitive graphs is treated in \autoref{sec:transitive}, while
the appendix collects some proofs
of known results
for the convenience of the reader.

\section{Graph Notation and the Laplacian}\label{sec:prelim}
Let $G=(V,E)$ be a finite or infinite, weighted, undirected, connected
graph with more than one vertex. 
Since we allow weights, we do not allow multiple edges. We {\em do} 
allow loops. 
If $G$ is finite, we use $n:=|V|$ to denote the number of vertices. 
For each edge $(x,y)\in E$, let $\weight(x,y) > 0$ be the weight of $(x,y)$. 
In almost all instances,
throughout the paper we assume that $\weight(x,y)\geq 1$
for every edge $(x,y)\in E$. However, we make this assumption explicit each
time.  We say $G$ is \dfn{unweighted} if $\weight(x,y)=1$
for every edge $(x,y)\in E$.

For each vertex $x\in V$, let $\weight(x) :=\sum_{y\sim x} \weight(x,y)$ be
the (weighted) degree of $x$ in $G$. Since $G$ is connected, $w(x)>0$ for
all $x\in V$. We always assume that our graph is such that $\weight(x) <
\infty$ for each of its vertices, $x$.
For a set $S\subseteq V$, let $\vol(S):=\sum_{x\in S}
\weight(x)$.
Similarly, let $\vol(E') := \sum_{e \in  E'} w(e)$ for $E' \subseteq E$.
We define the function
$\pi\colon V \to \R$ by 
$\pi(x):=\weight(x)/\vol(V)$.

For a vertex $x\in V$, we use $\b1_x$ to denote the indicator vector of $x$,
$$\b1_x(y):=\begin{cases}1& \text{if } y=x,\\0&\text{otherwise}.\end{cases}$$
We also use $\be_x := \b1_x/\sqrt{\weight(x)}$. For two vertices
$x,y\in V$, we use $\dist(x,y)$ to denote the length of a shortest path from $x$ to $y$. 
For every $r\geq 0$, we write $B_{\dist}(x,r):=\{y: \dist(x,y)\leq r\}$ to denote the set of vertices at distance at most $r$ from $x$. 
Define $\diam(x) := \sup_y \dist(x, y)$ and $\diam := \max_x \diam(x)$.
For a vertex $x\in V$ and radius $r\geq 0$, let 
$$
\vol(x,r):=\vol\bigl(B_{\dist}(x,r)\big) := \sum_{y\, :\, \dist(x,y)\leq r}
\weight(y)\,.
$$

We write $\ell^2(V, \weight)$ for the (real or complex)
Hilbert space of functions $f \colon V \to \R \mbox{ or } \C$ with
inner product
$$\langle f,g \rangle_{\weight} := \sum_{x \in V} \weight(x) f(x)
\overline{g(x)}$$
and squared norm $\|f\|_{\weight}^2 := \langle f,f\rangle_{\weight}$.
Note that an orthonormal basis of $\ell^2(V, \weight)$ is formed by the
vectors $\be_x$ ($x \in V$).
We reserve $\langle \cdot , \cdot \rangle$ and $\|\cdot\|$ for the standard
inner product and norm on $\mathbb R^k$, $k \in \mathbb N$ and 
$\ell^2(V)$, and also for the norm on a generic Hilbert space.

The \dfn{transition operator} $P \colon \ell^2(V, \weight) \to
\ell^2(V, \weight)$ is defined by 
\[
(Pf)(x) 
:=
\sum_{y \in V} \frac{w(x, y)}{\weight(x)} f(y)
\,;
\]
it is easily checked to have norm at most 1 and to be self-adjoint.
The \dfn{probabilistic Laplacian} is $\cL := I - P$.

It is well known and easy to check that
for $f \in \ell^2(V, \weight)$, we have
\[
\langle f, \cL f\rangle_\weight
=
\sum_{x \sim y} w(x, y) |f(x)-f(y)|^2
\,.
\]
Note that the sum over $x \sim y$ is over unordered pairs, i.e., over all
undirected edges.
Thus
for $f \in \ell^2(V, \weight)$ other than the zero function, we call
\[
\frac{\langle f, \cL f\rangle_\weight}{\langle f, f\rangle_{\weight}} 
=
\frac{\sum_{x \sim y} w(x, y) |f(x)-f(y)|^2}{\sum_{x \in V} \weight(x)
|f(x)|^2} =: \cR_G(f)
\]
the  \dfn{Rayleigh quotient of $f$ (with respect to $G$)}.
If $f$ takes values in a Hilbert space, then we define $\cR(f)$ similarly
with absolute values replaced by norms.

In particular, when $G$ is finite,
one sees that $\cL$ is a positive semi-definite operator with
eigenvalues $$0 = \lambda_1 < \lambda_2 \leq \cdots \leq \lambda_n \leq 2\,.$$
Here, $\lambda_1 < \lambda_2$ since $G$ is connected, which implies
that the first eigenvalue corresponds only
to the constant eigenfunctions.
Furthermore, by standard variational principles,
\begin{equation}
\lambda_k 
= \min_{H \subseteq \ell^2(V,\weight)} \max_{f \in H \setminus \{ \b0\}}
\vphantom{\bigoplus} \cR_G(f) \,,
\label{eq:eigenvar}
\end{equation}
where the minimum is over subspaces of dimension $k$.

\section{Spectral Measure and Spectral Embedding} \label{sec:spectrals}

\subsection{Spectral Measure}
The spectral theory of the Laplacian generalizes naturally to infinite
graphs. However, eigenvalues may not exist. Instead,
one defines probability measures on the spectrum; in the finite-graph
case, this amounts to assigning equal weight to each eigenvalue.
Here we describe  the spectral theory of infinite locally finite
graphs that may be equivalently applied to finite graphs. 

We begin with a brief review of
the spectral theorem for bounded self-adjoint operators
$T$ on a complex Hilbert space $\cH$. For more details, see, e.g.,
\cite[Chap.~12]{Rudin}.
Let $\cB$ be the Borel $\sigma$-field in $\R$.
A \dfn{resolution of the identity} $I(\cdot)$ is a map from $\cB$ to the
space of orthogonal projections on $\cH$ that satisfies properties similar
to a probability measure, namely, $I(\varnothing) = 0$; $I(\R) = I$; for all
$B_1, B_2 \in \cB$, we have $I(B_1 \cap B_2) = I(B_1) I(B_2)$ and, if $B_1 \cap B_2 =
\varnothing$, then $I(B_1 \cup B_2) = I(B_1) + I(B_2)$; and for all $f, g \in \cH$,
the map $B \mapsto \langle I(B)f, g \rangle$ is a finite complex measure on
$\cB$. 
Note that $B \mapsto \iprod{I(B) f, f} = \norm{I(B)f}^2$ is a positive
measure of norm $\norm{f}^2$.
The \dfn{spectrum} of $T$ is the set of $\lambda \in \R$ such that $T -
\lambda I$ does not have an inverse on $\cH$.
The spectral theorem says that there is a unique resolution of the
identity, $I_T(\cdot)$, such that $T = \int \lambda\,dI_T(\lambda)$ in the
sense that for all $f, g \in \cH$, we have $\langle Tf, g \rangle = \int
\lambda\,d\langle I_T(\lambda)f, g \rangle$. Furthermore, $I_T$ is
supported on the spectrum of $T$. For a bounded Borel-measurable function
$h \colon \R \to \R$, one can define (via what is called the \dfn{symbolic
calculus}) a bounded self-adjoint operator $h(T)$
by $h(T) := \int h(\lambda) \,dI_T(\lambda)$, with integration meant in the
same sense as above. The operator $h(T)$ commutes with $T$. For example,
$\b1_B(T) = I_T(B)$.
In our case, $T$ will be positive semi-definite, whence its spectrum will
be contained in $\R^+$. In this case, we will write $I_T(\delta) :=
I_T\bigl([0,
\delta]\big)$ for its cumulative distribution function ($\delta \ge 0$).

For example, if $\cH = L^2(X, \mu)$ and $g \in L^\infty(X, \mu)$, then
the multiplication operator $M_g$ defined by
$M_g \colon f \to g\cdot f$ is a bounded linear transformation, which is
self-adjoint when $g$ is real.
In this case, $I_{M_g}(B) = M_{\b1_{g^{-1}[B]}}$ and $h(M_g) = M_{h \circ g}$.

If $\cH$ is finite dimensional and $T$ has spectrum $\sigma$,
one could alternatively write $I_T(B)
= \sum_{\lambda \in \sigma \cap B} P_\lambda$, where $P_\lambda$ is the
orthogonal projection onto the $\lambda$-eigenspace. 
In particular, $I_T(\delta) = \sum_{\lambda \le \delta} P_\lambda$.
Writing $T =
\sum_{\lambda \in \sigma} \lambda P_\lambda$ amounts to
diagonalizing $T$.
Here we have $h(T) =
\sum_{\lambda \in \sigma} h(\lambda) P_\lambda$ for any function $h$;
because only finitely many values of $h$ are used, we may take $h$ to be a
polynomial. 

Let $G$ be a locally finite graph.
Let $I_\cL(\cdot)$ be the resolution of the identity for the operator $\cL$. 
The Laplacian $\cL$ is a positive semi-definite self-adjoint operator acting
on $\ell^2(V, \weight)$ with operator norm at most 2, so its spectrum is contained
in $[0, 2]$.
We may use $I(\cdot)$ whenever the operator is clear from context. 
For a vertex $x\in V$ and $\delta > 0$, the function 
\begin{equation}
\label{eq:vertexspectralmeasure}
\mu_x(\delta) := \langle I_\cL(\delta)\be_x, \be_x\rangle_\weight
= \langle I_\cL(\delta)\b1_x, \b1_x\rangle
\end{equation}
is called  the  \dfn{vertex spectral measure} of $x$. 
It defines a probability measure on the Borel sets of $\mathbb{R}$
supported on $[0, 2]$.
If $G$ is finite, then the  \dfn{spectral measure} of $G$ is defined as
\begin{equation}
\label{eq:spectralmeasure}
\mu(\delta):=\frac{1}{n} \sum_{x\in V} \mu_x(\delta)
= \frac{1}{n}|\{k : \lambda_k \le \delta\}|
\,. 
\end{equation}
For general infinite graphs, there is no corresponding spectral measure,
other than the projection-valued $I_\cL$. Of
course, if $G$ is transitive, then $\mu_x(\delta)$ does not depend on $x
\in V$, and in this case, $\mu_x$ is already an analogue of $\mu$.

For infinite graphs with infinite volume, there is no kernel of $\cL$, so
$I(0) = \b0$.
However, for finite graphs, $I(0)$ is the projection on the kernel of
$\cL$, which is the space of constant functions.
Since we are not interested in the kernel of $\cL$,
it will be convenient for us to work with the
operator $\Iz(\delta):=I(\delta)-I(0)$. 
Correspondingly, we define 
\begin{equation}
\label{eq:vertexspectralmeasurestar}
\muz_x(\delta):=\langle \Iz(\delta)\b1_x,\b1_x\rangle\,, ~~~~
\muz(\delta):=\frac{1}{n} \sum_{x\in V} \muz_x(\delta)\,.
\end{equation}
Recall the definition $\pi(x) := \weight(x)/\vol(V)$ from the beginning of
\autoref{sec:prelim}.
Observe that  $\muz_x(\delta) = \mu_x(\delta)-\pi(x)$. 
Therefore, for every connected graph $G$ and every vertex
$x\in V$, we have $\muz_x(0)=0$ and $\muz_x(2)=1-\pi(x)$. Furthermore,
$\muz(\delta)=\mu(\delta)-1/n$ when $G$ is finite.

\begin{example}
\label{ex:cycle}
Consider the unweighted cycle on $n$ vertices, which we regard as
the usual Cayley graph of $\Z_n := \Z/n\Z$.
Let $\nu$ be the uniform probability measure on $\Z_n$.
The Fourier transform $\FT$ maps $\ell^2(\Z_n)$ isometrically isomorphically to
$\ell^2(\Z_n, \nu)$ and carries $\cL$ to the multiplication operator $M_g$,
where $g(k) := 1 - \cos(2\pi k/n)$.
The eigenvalues of $\cL$ are the values (with multiplicity) of $g$.
Since $\cL = \FT^{-1} M_g \FT$, we have that $I_\cL = \FT^{-1} I_{M_g}
\FT$. Clearly,
$I_{M_g}[0, \lambda] = M_{\b1_{B_\lambda}}$, where $B_\lambda :=
\{k : g(k) \le \lambda\}$.
Therefore, $\mu_x(\delta) = |B_\delta|/n$ for all $x\in \Z_n$.
\end{example}

\begin{example}
\label{ex:line}
Consider the usual unweighted Cayley graph of $\Z$.
The Fourier transform $\FT$ maps $L^2(\R/\Z)$ isometrically isomorphically to
$\ell^2(\Z)$ and carries the multiplication operator $M_g$ to $\cL$, where
$g(s) := 1 - \cos(2\pi s)$.
Since $\cL = \FT M_g \FT^{-1}$, we have that $I_\cL = \FT I_{M_g}
\FT^{-1}$. Clearly,
$I_{M_g}[0, \lambda] = M_{\b1_{B_\lambda}}$, where $B_\lambda :=
\{s : g(s) \le \lambda\}$.
Therefore, $\mu_x(\delta) = |B_\delta|$ for all $x\in \Z$.
\end{example}

It is straightforward that characterizing spectral measure of finite graphs
provides a corresponding characterization for the eigenvalues of the normalized Laplacian. 
\begin{fact}
\label{lem:lambdamutrans}
For every finite graph $G$ and $2\leq k\leq n$, if $\muz(\delta)\leq (k-1)/n$, then $ \lambda_k \geq \delta$.
\end{fact}

The next lemma is a generalization of the Rayleigh quotient to infinite graphs.
\begin{lemma}
\label{lem:rayleighquotient}
Let $f\in \ell^2(V, \weight)$ and $\delta\in [0,2]$. If $f\in
\img\bigl(I(\delta)\big)$, then
$$ \langle \cL f, f\rangle_\weight \leq \delta\langle f,f\rangle_\weight = \delta
\sum_{x\in V} \weight(x) f(x)^2\,.$$
\end{lemma}
\begin{proof}
 Since $f\in \img\bigl(I(\delta)\big)$, 
 if $B \cap [0, \delta] = \varnothing$, then
\[
\langle I_\cL f,f\rangle_\weight (B) =
\langle I_\cL I(\delta) f,f\rangle_\weight (B)
=\langle I(\delta)(B) f,f\rangle_\weight 
=
\iprod{\b0 f, f}_\weight = 0
\,.
\]
That is, $\iprod{I_\cL f, f}_\weight$ is supported on $[0, \delta]$.
Therefore,
\begin{equation*} 
\langle \cL f,f\rangle_\weight 
=
\int_{[0, 2]} \lambda \, d\langle I(\lambda) f,f\rangle_\weight 
= \int_{[0,\delta]} \lambda \,
d\langle I(\lambda)f,f\rangle_\weight 
\le \delta\langle f,f\rangle_\weight\,.
\qedhere
\end{equation*}
\end{proof}

\subsection{Random Walks} \label{subsec:ran}
A random walk is called \dfn{lazy} if for each
vertex $x$, we have $p(x, x) \ge 1/2$.
This guarantees aperiodicity of the random walk and also that $P$ is
positive semi-definite.
Recall that $\cL = I - P$ whether or not $P$ is
lazy; the eigenvalues of $\cL$ lie in $[0, 1]$ when
$P$ is lazy and in $[0, 2]$ in all cases. 
If $G$ is a loopless unweighted graph, then \dfn{lazy simple random walk}
on $G$ is the random walk on the graph $G'$ obtained from $G$ by adding
$\weight(x)$ loops to each vertex $x$.
In this case, if $P$ is the transition matrix for $G$ and $P'$ that for
$G'$, we have $P' = (I+P)/2$, whence the corresponding Laplacians satisfy
$\cL' = \cL/2$ and $I_{\cL'}(\lambda) = I_{\cL}(2\lambda)$.

An alternative to laziness
is \dfn{continuous-time random walk}, which has the
transition matrix $P$, but rather than take steps at each positive
integer time, it takes steps at the times of a Poisson process with rate 1.
In other words, the times between steps are IID random variables with
Exponential(1) distribution. These random walks behave very similarly to
the discrete-time lazy random walks.
The mathematics is slightly cleaner for continuous time
than for discrete time; sometimes one can derive bounds for one from bounds
for the other, but often it is easier simply to follow the same proof.

Obviously for every finite graph $G$, the $k$th largest
eigenvalue of  $P$ is equal to $1$ minus the $k$th smallest
eigenvalue of $\cL$. That is, the eigenvalues  of $P$ are
$$ 1=1-\lambda_1\geq 1-\lambda_2\geq \ldots\geq 1-\lambda_n\geq -1\,.$$

For two vertices $x,y\in V$,  we use $p_t(x,y)$ to denote the probability
that the discrete-time random
walk started at $x$ arrives at $y$ at step number $t$. Observe
that $p_t(x,y)=\langle P^t\b1_y, \b1_x\rangle$. 
For a finite, connected graph $G$, let $\pi(\cdot)$ be the stationary
distribution of the walk. It is elementary that $\pi(x) =
\weight(x)/\vol(V)$ for all $x\in V$.  For every $p>0$ and $\eps>0$, the
\dfn{$L^p$-mixing time} of the walk is defined as
$$ \tau_p(\eps) := \min\left\{t: \forall x\in V\ \left(\sum_{y\in V}
\left|\frac{p_t(x,y)}{\pi(y)}-1\right|^{p}\pi(y)\right)^{1/p}\leq
\eps \right\}\,.$$
For $p = \infty$, one defines
$$ \tau_\infty(\eps) := \min\left\{t: 
\forall x, y\in V\ \left|\frac{p_t(x,y)}{\pi(y)}-1\right|\leq
\eps \right\}\,.$$
It is elementary that 
for every $\eps>0$,
\begin{equation}
\label{eq:mixingreturnprob}
 \ceil{\tau_\infty(\eps)/2} = \tau_2(\sqrt\eps) = \min\left\{ t:
 \forall x\in V\ \> \frac{p_{2t}(x,x)}{\pi(x)}\leq 1+\eps \right\}.
 \end{equation}
We present a self-contained proof in \autoref{prop:L2returnprob}.

We use $q_t(x, y)$ for the probability that the continuous-time random walk
started at $x$ is at $y$ at time $t$. 
This is also known as the \dfn{heat kernel} on $G$.
We have 
\begin{equation}
\label{eq:generator}
q_t(x, y)
=
\iprod{e^{-t \cL} \b1_y, \b1_x}
=
\int_0^2 e^{-\lambda t} \,d\iprod{I_\cL(\lambda) \b1_y, \b1_x}
\,.
\end{equation}
One defines $L^p$-mixing times for continuous-time random walks in the same
way as for discrete-time random walks.
In this case, \eqref{eq:mixingreturnprob} holds without the ceiling signs.

We can use the spectral measure of the Laplacian to upper-bound the return
probability, or the mixing time, of the random walks. 
Recall that when $G$ has infinite volume, $\pi(x) := 0$.
\begin{lemma}
\label{lem:returnprobspectral}
Consider a lazy random walk on a weighted graph, $G$.
If %
$\muz_x(\lambda) \le \psi(\lambda)$ for some increasing continuously
differentiable function $\psi$ with $\psi(0) = 0$, then
\[
p_t(x,x) - \pi(x)
=
t\int_0^1(1-\lambda)^{t-1}\muz_x(\lambda)\,d\lambda 
\le  \int_0^1(1-\lambda)^t\psi'(\lambda) \, d\lambda
\leq  \int_0^1 e^{-\lambda t} \psi'(\lambda) \,
\,d\lambda\,.
\]
Hence if $G$ is finite
and
$\muz(\lambda)\leq \psi(\lambda)$ for some increasing continuously
differentiable function $\psi$ with $\psi(0) = 0$, then
$$ \frac{\sum_{x \in V(G)} p_t(x,x) - 1}{n} 
=
t\int_0^1(1-\lambda)^{t-1}\muz(\lambda)\,d\lambda 
\leq \int_0^1 (1-\lambda)^t
\psi'(\lambda)\,d\lambda  \leq  \int_0^1 e^{-\lambda t}
\psi'(\lambda)\, d\lambda\,.$$
\end{lemma}
\begin{proof} First, since $P= I-\cL$, we have
\begin{equation*}
p_t(x,x) 
=  \langle (I-\cL)^t\b1_x,\b1_x\rangle\,. 
\end{equation*}
Symbolic calculus gives
$$ (I-\cL)^t = \int_{0}^1 (1-\lambda)^t\, dI_{\cL}(\lambda)\,.$$
Therefore, by \eqref{eq:vertexspectralmeasure}, we get
\begin{align*} p_t(x,x) = \int_{0}^1 (1-\lambda)^t \, d\langle
I_\cL(\lambda)\b1_x,\b1_x\rangle 
&= \pi(x) + \int_{0}^1(1-\lambda)^t \, d\muz_x(\lambda) \\
&= \pi(x) + t\int_{0}^1(1-\lambda)^{t-1}\muz_x(\lambda)\, d\lambda
\,,
\end{align*}
where the third equation holds by the fact that $\muz_x(0)=0$. Thus, if
$\muz_x(\lambda)\leq \psi(\lambda)$ and $\psi(\cdot)$ is continuously
differentiable, 
\[
p_t(x,x) - \pi(x) = t\int_0^1 (1-\lambda)^{t-1}
\muz_x(\lambda) \,d\lambda 
\leq
t\int_0^1(1-\lambda)^{t-1} \psi(\lambda)\,d\lambda\\
= \int_0^1(1-\lambda)^t\psi'(\lambda)\, d\lambda\,.
\qedhere
\]
\end{proof}

For continuous-time random walk, \eqref{eq:generator} tells us that
the return probability is given by the Laplace transform, i.e.,
\[
q_t(x, x) - \pi(x) 
=
\int_0^2 e^{-\lambda t} \,d\muz_x(\lambda)
\,.
\]
This makes formulas somewhat cleaner. But as we used in the preceding proof,
$p_t(x, x) < q_{t}(x, x)$ for $t \ge 1$.

An upper bound on spectral measure gives an upper bound on return
probabilities, as in \autoref{lem:returnprobspectral}.
The reverse is true as well, as noted by \cite[Proposition 5.3]{EfShu}:

\begin{lemma}
\label{lem:reversereturnprobspectral}
Consider random walk on a weighted graph $G$.
For every vertex $x\in V$, if the walk is lazy, then
\[
\muz_x(\delta)
\le
2e \cdot \bigl(p_{\flr{1/\delta}}(x, x) - \pi(x)\big)
\quad\mbox{for}\quad
0 < \delta \le 1/2
\,,
\]
while even if the walk is not lazy,
\[
\muz_x(\delta)
\le
e \cdot \bigl(q_{1/\delta}(x, x) - \pi(x)\big)
\quad\mbox{for}\quad
0 < \delta \le 2
\,.
\]
\end{lemma}

See the appendix for a proof.

In order to show that certain results are sharp, it is useful to see how
having both an upper and a lower bound on return probabilities gives a lower
bound on spectral measure. 
The following is again due to \cite[Proposition 5.3]{EfShu}.

\begin{lemma}
\label{lem:bothreturnprobspectral}
Consider random walk on a weighted graph $G$.
For every vertex $x\in V$, if the walk is lazy, then
\[
\muz_x(\delta)
\ge
\bigl(p_{t}(x, x) - \pi(x)\big) - (1 - \delta)^{\flr{t/2}}
\bigl(p_{\ceil{t/2}}(x, x) - \pi(x)\big)
\quad\mbox{for}\quad
0 < \delta \le 1
\quad\mbox{and}\quad
t \ge 1
\,,
\]
while even if the walk is not lazy,
\[
\muz_x(\delta)
\ge
\bigl(q_{t}(x, x) - \pi(x)\big) - e^{-\delta {t/2}}
\bigl(q_{{t/2}}(x, x) - \pi(x)\big)
\quad\mbox{for}\quad
0 < \delta \le 2
\quad\mbox{and}\quad
t > 0
\,.
\]
\end{lemma}

See the appendix for a proof.

For more information on polynomial-decay asymptotics, comparing
$\muz_x(\delta)$ for small $\delta$ with $p_t(x, x) - \pi(x)$ for large
$t$, see \cite[Appendix 1]{GroShu}.

We will generally state our results only for discrete-time random walks,
but analogous results follow from similar proofs for continuous time.

\subsection{Embeddings of Graphs}
We start by describing general properties of every embedding of a 
graph $G$ into a (real or complex) Hilbert space $\cH$.
Usually, we will use $\cH = \ell^2(V, \weight)$.
Let $F\colon V\rightarrow \cH$ be an embedding of $G$. (Note that by
``embedding", we do not imply that $F$ is injective; it is merely a map.)
We say that $F$ is
\dfn{\spheresym} if $G$ is finite and
$\sum_{x\in V} F_x w(x)=\b0$, i.e., $F \perp \b1$ in $\ell^2(V, \cH,
\weight)$.
We also say that $F$ is  \dfn{non-trivial} if $F_x\neq \b0$ for some $x\in V$.
For a vertex $x\in V$ and radius $r \geq 0$, we use $B_F(x,r):=\big\{y\in
V:\norm{F_x-F_y}\leq r\big\}$ to denote the set of vertices of $G$
mapped to a ball of $\cH$-radius $r$ about $F_x$.

For a subset $E'$ of edges of $G$, we define the  \dfn{energy} of $F$ on
$E'$ as
\begin{equation}
\cE_F(E'):=\sum_{(x,y)\in E'} \weight(x,y) \norm{F_x - F_y}^2\,.
\end{equation}
Roughly speaking, the energy of a subgraph of $G$ describes the stretch of the edges of that subgraph under the embedding $F$.
For a set $S\subseteq V$ of vertices, we define the  \dfn{energy} 
$\cE_F(S)$ of $S$ as the energy of all edges with at least one endpoint in
$S$.

If we use the weight of an edge as its conductance, then we can
relate energies to
effective resistances of the corresponding electrical network: For a finite
graph $G$, we define the  \dfn{effective conductance} between
a pair of vertices $a,z\in V$ as 
$$\ceff(a,z) := \min_{f(a) \ne f(z)}\frac{\cE_f(E)}{|f(a)-f(z)|^2}
=
\min_{f(a) \ne f(z)}
\frac{\sum_{x\sim y}\weight(x,y)|f(x)-f(y)|^2}{|f(a)-f(z)|^2}\,.$$
This is for scalar-valued functions $f$, but by adding the squares of
coordinates, the same holds for vector-valued functions in place of
$f$ and with norms in place of absolute values.
The \dfn{effective resistance} $\reff(a, z)$ is the reciprocal of the effective
conductance.
We also use $\rdiam:=\sup_{a,z\in V} \reff(a,z)$ to denote the maximum
effective resistance of any pair of vertices of $V$, the \dfn{effective
resistance diameter} of $G$. 
Similarly, define
$
\rdiam(x) := \max_z \reff(x, z)
$. 
It is well known that the expected time for the random walk to go
from $a$ to $z$ and then back to $a$ is equal to $\vol(V) \reff(a, z)$;
this is called the \dfn{commute time} between $a$ and $z$. The maximum
commute time between $x$ and any other vertex will be denoted $\tcommute^x
:= \vol(V) \rdiam(x)$, while the maximum commute time between any pair of
vertices will be denoted $\tcommute^* := \vol(V) \rdiam$.  
We refer to \cite[Chap.~2]{LP16} for more
background on electrical networks and their connections to random walks.

The following lemmas are used in several of our proofs.

\begin{lemma}
\label{lem:spheresymball}
For every non-trivial \spheresym~embedding $F \colon V\rightarrow \cH$
of a finite graph $G$,  we have $B_F\bigl(x, \norm{F_x}\big)\neq V$ for all
$x \in V$. 
\end{lemma}
\begin{proof}
 Suppose that $B_F\bigl(x,\norm{F_x}\big)=V$ for some $x \in V$.
 Then for every vertex $y\in V$, we have $\big\langle F_x, F_y
 \big\rangle\geq 0$.
 Since $F$ is \spheresym, we have
$$ 0 = \sum_{y\in V} w(y) \big\langle F_x, F_y \big\rangle \geq w(x)
\big\langle F_x, F_x\big\rangle \,, $$
whence $\norm{F_x} = 0$. Therefore $F_y = F_x$ for all $y \in V$. Since $F$
is \spheresym, it follows that $F$ is trivial.
\end{proof}

\begin{lemma}
\label{lem:pathenergy}
Suppose that $w(x, y) \ge 1$ for all edges $(x, y) \in E$.
Let $F \colon V \rightarrow \cH$ be any embedding of $G$ into a Hilbert
space $\cH$, and let $B:=B_F(x,r)$. If $\cP\subseteq E$ is a simple path in $G$
starting at $x$ whose last vertex only is outside of $B$, then
$$ \cE_F(B) \ge \cE_F(\cP) \geq \frac{r^2}{|\cP|}\geq \frac{r^2}{|B|}\,.$$
\end{lemma}
\begin{proof}
Let $\cP=(y_0,y_1, y_2,\ldots, y_{l-1},y_l)$, where $y_0 = x$ and
$y_i\notin B$ iff $i = l = |\cP|$.
Then by the arithmetic mean-quadratic mean inequality, we have
\begin{align*}
 \cE_F(\cP) &\geq \sum_{i=0}^{l-1} \norm{F_{y_i} - F_{y_{i+1}}}^2 \geq
 \frac{1}{l} \left(\sum_{i=0}^{l-1} \norm{F_{y_i} - F_{y_{i+1}}}\right)^2
 \\ &\geq \frac{1}{l} \norm{F_{y_0} - F_{y_l}}^2 \geq \frac{r^2}{|\cP|}\,,
 \end{align*}
where the first inequality uses the assumption that $w(x,y)\geq 1$ for all
edges $(x,y)\in E$, the third inequality follows by the triangle inequality
in Hilbert space, and the last inequality follows by the assumption
that $y_l\notin B$.
\end{proof}

\subsection{Spectral Embedding}
\label{subsec:spectralembedding}
For $\delta\in (0,2)$, we define the \dfn{spectral embedding} $F \colon
V \to \ell^2(V, \weight)$ of $G$ by
$$x \mapsto F_x:= \Iz(\delta) \be_x/\sqrt{\weight(x)}
= \Iz(\delta)\b1_x/{\weight(x)}\,.$$
Our notation does not reflect the dependence of $F$ on $\delta$;
in all cases, $\delta$ will be fixed when $F$ is used.
Since $\cL = \overline \cL$, it follows that $F_x$ is real valued for all
$x$.

For finite graphs, another way to view this embedding is as follows.
Suppose that $\delta = \lambda_k \ne \lambda_{k+1}$ and that 
$g_1,g_2,\ldots,g_n \colon  V\rightarrow \mathbb{R}$ is an
orthonormal basis of $\ell^2(V, \weight)$ such that $g_2,\ldots,g_k$ 
span the image of $\Iz(\delta)$.
For example, $g_j$ could be a $\lambda_j$-eigenvector of $\cL$. 
Because $F_x$ lies in $\img\bigl(\Iz(\delta)\big)$,
we may write $F_x$ in the $(g_1, \ldots, g_n)$-coordinates as 
$$F_x = \bigl(0, f_2(x), \ldots,f_k(x), 0, \ldots, 0\big).$$ 
In order to calculate $f_j(x)$, we write
\begin{align*}
f_j(x) &= \langle F_x, g_j \rangle_\weight
= \frac{\langle \Iz(\delta) \b1_x, g_j \rangle_\weight}{{w(x)}}
= \frac{\langle \b1_x, \Iz(\delta) g_j \rangle_\weight}{{w(x)}}
\\ &= \frac{\langle \b1_x, g_j \rangle_\weight}{{w(x)}}
= {\langle \b1_x, g_j \rangle}
= {\overline{g_j(x)}} = g_j(x)\,.
\end{align*}
One could, therefore, work simply with $\bigl(g_2(x), \ldots, g_k(x)\big)$, 
and translate all our proofs for finite graphs into such language.
For infinite graphs, one could use infinitely many vectors $g_j$, but
they would not be eigenvectors.

\begin{lemma}
\label{lem:spheresym}
For every finite graph $G$, the above $F$ is \spheresym. 
\end{lemma}
\begin{proof}
This is clear from the fact that $\img\bigl(\Iz(\delta)\big) \perp \b1$.
\end{proof}

\begin{lemma}
\label{lem:normF}
For every finite or infinite graph $G$ and every vertex $x\in V$, 
$$
\norm{F_x}^2_\weight
= F_x(x)=\muz_x(\delta)/\weight(x)\,.
$$ 
Hence, for every finite graph $G$,
$$\sum_{x\in V} \norm{F_x}_\weight^2\weight(x)=\sum_{x\in V} \muz_x(\delta) =
n\muz(\delta)\,.$$
\end{lemma}
\begin{proof}
The definitions of spectral embedding and
spectral measure \eqref{eq:vertexspectralmeasurestar} give
\[
F_x(x) = \frac{\langle
\Iz(\delta)\b1_x,\b1_x\rangle}{{\weight(x)}} =
\frac{\muz_x(\delta)}{{\weight(x)}}
\]
and
\[
\norm{F_x}_\weight^2
= \frac{\langle  \Iz(\delta)\b1_x,
\Iz(\delta)\b1_x\rangle_\weight}{\weight(x)^2}
= \frac{\langle \Iz(\delta)\b1_x, \b1_x\rangle_\weight}{\weight(x)^2}
= \frac{\langle \Iz(\delta)\b1_x, \b1_x\rangle}{\weight(x)}
\,.
\qedhere
\]
\end{proof}

\begin{lemma}
\label{cl:fprops}
If $\muz_x(\delta) > 0$, define $f \colon V\rightarrow\mathbb{C}$ by
$$f:= \frac{F_x}{\norm{F_x}_\weight} \,.$$ 
Then
\begin{enumerate}[\rm i)]
\item $\norm{f}_\weight=1,$
\item $f(x) = \sqrt{\muz_x(\delta)/\weight(x)},$
\item $f\in \img\bigl(\Iz(\delta)\big)$.
\end{enumerate}
\end{lemma}
\begin{proof}
The first and third parts are obvious, while the second follows from
\autoref{lem:normF}.
\end{proof}

For a set $S\subseteq V$, let $\muz_S(\delta):=\sum_{x\in S} \muz_x(\delta)$. The proof of the next lemma is based on  \cite[Lemma 3.2]{LOT12}.
\begin{lemma}
\label{lem:isotropy}
For every $\delta\in (0,2)$, the spectral embedding $F$ enjoys the following properties:
\begin{enumerate}[\rm i)]
\item For every $f\in \img\bigl(\Iz(\delta)\big)$ with $\norm{f}_\weight =
1$, we have
$ \sum_{x\in V}  \weight(x) \big|\big\langle f,F_x\big\rangle_\weight\big|^2 = 1$.
\item For every vertex $x\in V$  and $r:=\alpha \norm{F_x}_\weight$ with $0
< \alpha < 1/\sqrt2$,
we have
$$\muz_{B_F(x,r)}(\delta)  \leq \frac{1}{(1-2\alpha^2)^2}\,.$$
\end{enumerate}
\end{lemma}
\begin{proof}
First we prove (i):
$$ \sum_{x\in V} \weight(x) \big|\langle f,F_x\rangle_\weight\big|^2
= \sum_{x\in V}  \big|\langle f, \Iz(\delta)\be_x \rangle_\weight\big|^2
= \sum_{x\in V} \big|\langle \Iz(\delta)f,\be_x\rangle_\weight\big|^2 
= 1\,, $$
where the last equality follows by the fact that $\norm{f}_\weight=1$ and
$f = \Iz(\delta)(f)$.
It remains to prove (ii). First observe that for every two non-zero vectors
$f,g$ in any Hilbert space, we have
$$
\|f\| \left\|\frac{f}{\|f\|} - \frac{g}{\|g\|}\right\|
=
\left\|f - \frac{\|f\|}{\|g\|} g\right\|
\leq
\|f-g\| + \left\|g - \frac{\|f\|}{\|g\|} g\right\|
\leq 2 \,\|f-g\|\,.
$$
Therefore,
$$ \Re\left\langle \frac{f}{\norm{f}},\frac{g}{\norm{g}}\right\rangle =
\frac12\left(2-\norm{\frac{f}{\norm{f}}-\frac{g}{\norm{g}}}^2\right) \geq
1-2\frac{\norm{f-g}^2}{\norm{f}^2}\,.$$
Now let $f:=F_x/\norm{F_x}_\weight$. Since $f\in \img\bigl(\Iz(\delta)\big)$,  we
have by (i) and \autoref{lem:normF} that
\begin{align*} 1
&=\sum_{y\in V} \weight(y) \big|\left\langle F_y, f\right\rangle_\weight\big|^2
\geq \sum_{y \in B_F(x,r)} \weight(y) \norm{F_y}_\weight^2 \left|\left\langle
\frac{F_y}{\norm{F_y}_\weight},\frac{F_x}{\norm{F_x}_\weight}\right\rangle_{\!\!\weight}\right|^2 \\
&\geq \sum_{y \in B_F(x,r)} \muz_y(\delta)
\left(1-2\frac{\norm{F_x-F_y}_\weight^2}{\norm{F_x}_\weight^2}\right)^2 
\geq \muz_{B_F(x,r)}(\delta)\left( 1- 2\alpha^2\right)^2.
\qedhere
\end{align*}
\end{proof}

\begin{lemma}
\label{lem:rayleighquotientF}
For every finite graph $G$ and $\delta\in[\lambda_2,2)$,
$$\delta\geq \frac{\cE_F(V)}{\sum_{x} \weight(x) \norm{F_x}_\weight^2 }
= \cR(F)\,.$$
\end{lemma}
\begin{proof}
Note that 
\[
F_x(y)
=
\big\langle \Iz(\delta) \b1_x/\weight(x), \b1_y \big\rangle
=
\big\langle \Iz(\delta) \b1_x/\weight(x), \b1_y/\weight(y)
\big\rangle_\weight
=
F_y(x)
\]
for all $x, y \in V$. Therefore
\begin{align*}
{\cE_F(V)}
&=\sum_{y\sim z} \weight(y,z) \norm{F_y-F_z}_\weight^2
=\sum_{y\sim z} \weight(y,z)\sum_x \weight(x) \big|\bigl(F_y-F_z\big)(x)\big|^2
\\ &=\sum_{y\sim z} \weight(y,z)\sum_x \weight(x) |F_x(y)-F_x(z)|^2
= {\sum_x \weight(x) \langle F_x, \cL F_x\rangle_\weight}
\\ &\leq 
{\sum_{x} \delta\,\weight(x) \norm{F_x}_\weight^2}
\,,
\end{align*}
where the inequality holds by \autoref{lem:rayleighquotient} and that
$F_x\in\img\bigl(\Iz(\delta)\big)$ for each $x$. 
\end{proof}

\begin{example}
Consider again \autoref{ex:cycle} of
the unweighted cycle on $n$ vertices, which we regard as
the usual Cayley graph of $\Z_n := \Z/n\Z$.
Choose $\delta := \lambda_2 = 1 - \cos(2\pi/n)$.
We may calculate the embedding $F \colon \Z_n \to \ell^2(\Z_n)$ by
identifying $\b1_x \in \ell^2(\Z_n)$ with its image $\chi_x \colon k \mapsto
e^{2\pi i x k / n}$ under
the Fourier transform. (This is also the image of $\be_x$ under the
resulting isometric isomorphism from $\ell^2(V, \weight)$ to $\ell^2(\Z_n)$.)
Then $F_x \colon k \mapsto \b1_{\{\pm1\}}(k)
\chi_x(k)/\sqrt2$.
The image of $F$ is a set of $n$ points equally spaced on a circle.
\end{example}

\begin{example}
Consider again \autoref{ex:line} of the usual unweighted Cayley graph of $\Z$.
The embedding $F \colon \Z \to \ell^2(\Z)$ is easiest to perceive if we
identify $\ell^2(\Z)$ with $L^2(\R/\Z)$ (via
the Fourier transform). Then $F_x = \b1_{B_\delta} \chi_x/\sqrt2$,
where $\chi_x \colon s \mapsto e^{2\pi i x s}$, since $\chi_x$ is the image
of $\be_x$ under the isometric isomorphism from $\ell^2(V, \weight)$ to
$L^2(\R/\Z)$.
These points are on an infinite-dimensional sphere, with the inner product
between $F_x$ and $F_y$ being $\sin\bigl((x-y) \cos^{-1}(1 -
\delta)\big)/\bigl(2\pi(x-y)\big)$ for $x \ne y$ and $0 \le \delta \le 2$.
See \autoref{fig:embedZ2} for an illustration for $\Z^2$, rather than for
$\Z$.
\begin{SCfigure}[.85][ht]
\includegraphics[width=3.5in]{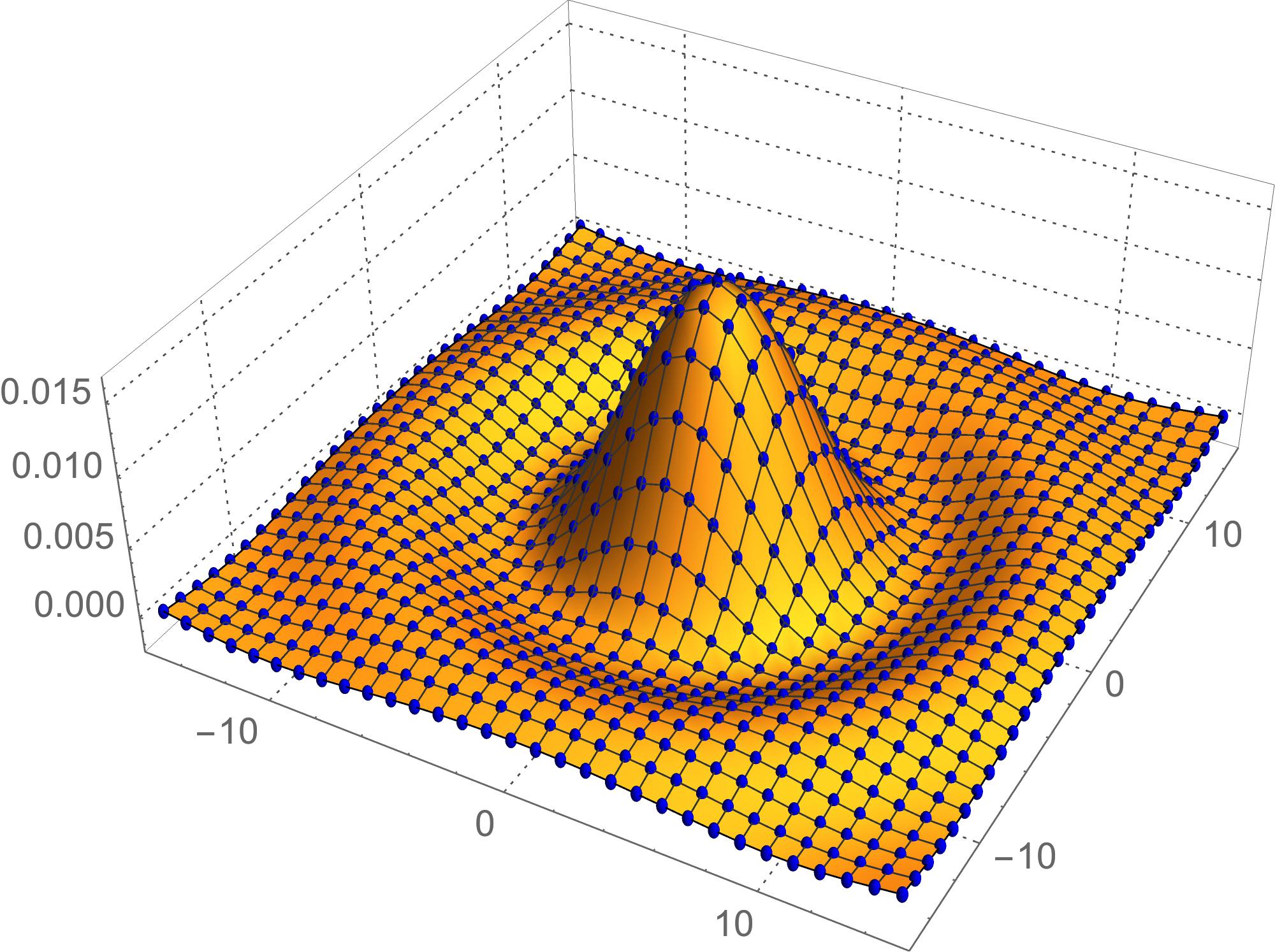}
\caption{The spectral embedding value $F_{(0, 0)} \in \ell^2(\Z^2)$ of
$\Z^2$ for $\delta = 0.1$. Here, $F_{(0,0)}$ is the 
Fourier transform of the indicator of $B_\delta$. The surface in the figure
is only to aid visualizing the values on $\Z^2$. The value
$F_{(j, k)}$ is the same but shifted to $(j, k)$.
\vspace{5pt}}
\label{fig:embedZ2}
\end{SCfigure}
\end{example}

 \section{Bounds on the Vertex Spectral Measure} \label{sec:lower}
Let $G$ be a locally finite graph.  
This section contains two subsections. In the first, we
treat worst-case finite graphs for eigenvalues,
spectral measure, return probabilities, and mixing.
In the second subsection, we treat the worst-case graphs when a lower bound
to the growth rate is imposed.
Both sections have results for regular graphs.

The structure of all our proofs follows the same two steps. In the first
step, we bound eigenvalues from below by the Rayleigh quotient of a
specially chosen function in the image of a spectral embedding. In the
second step, we bound the Rayleigh quotient from below via a geometric
argument.  The geometry will not enter in a serious way until the proof of
\autoref{thm:generalvertexmeasure}.
In general, when we bound the spectral measure $\muz_x(\delta)$ at a vertex
$x$, the embedding will place $x$ at a location whose distance from the
origin is related to $\muz_x(\delta)$. The energy of the embedding is then
bounded below by some version of the fact that other ``close"
vertices are embedded ``far" from the location of $x$. This fact, in turn,
arises from the property that the embedding is orthogonal to the kernel of
$\cL$. The meaning of ``far" depends on the assumptions of the theorem
desired.

\subsection{Worst-Case Finite Graphs}

We begin with a very simple proof of a lower bound on $\lambda_2$.
It shows that the relaxation time (i.e., $1/\lambda_2$) is bounded by half
the maximum commute time. This is well known; later we will improve it to
show that the $L^\infty$-mixing time is bounded by a constant times the
maximum commute time. In this proof, the first step (in the general
structure of our proofs) is trivial by choosing
an eigenfunction, and the second step is quite general.

\begin{proposition}
\label{thm:ltwo}
For every finite, connected, weighted graph $G$, we have
$$
\lambda_2 \geq \frac{2}{\tcommute^*}
\,.
$$
In particular, if $G$ is unweighted and loopless, then
$$
\lambda_2 \geq \frac{2}{n(n-1)^2}
\,.
$$
\end{proposition}

\begin{proof}
Let $f$ be a unit-norm $\lambda_2$-eigenvector of $\cL$. 
Then 
$$
\lambda_2 = \cE_f(E)
\ge
|f(x) - f(y)|^2\, \ceff(x, y)
$$
for all $x, y \in V$.
Since $f \perp \b1$, we obtain
\[
\rdiam \lambda_2
\ge
\sum_{x, y} w(x) w(y) |f(x) - f(y)|^2/\vol(V)^2
=
\frac{2\vol(V) - 2 \big| \sum_x {w(x)} f(x)\big|^2}{\vol(V)^2}
=
\frac{2}{\vol(V)}
\,,
\]
as desired.
In the unweighted loopless case, we use the fact that $\vol(V) \le n(n-1)$ and
$\rdiam \le \diam \le n-1$ (as in \autoref{lem:pathenergy}).
\end{proof}

The maximum commute time is known (see \cite{CFS} for a simple proof)
to be at most $4n^3/27 + o(n^3)$
if $G$ is unweighted and loopless, whence we have the better bound
\begin{equation}
\label{eq:274bound}
\lambda_2 \geq \frac{27 + o(1)}{2n^3}
\end{equation}
in that case.

As is well known, this bound is sharp in various ways up to a constant
factor. For example, \cite{LO81} show that
the barbell graph, which has $\flr{n/3}$ vertices in
each of two cliques and $n-2\flr{n/3}$ vertices in a path that joins the
two cliques, has $\lambda_2 \le 54/n^3 +O(1/n^4)$.

One can regard the preceding proof as using the 1-dimensional embedding $f
\colon V \to \R$.
In the rest of the paper, we use higher-dimensional embeddings $F$ to
bound the spectral measure at a vertex.
However, in this section we still use only a 1-dimensional relative of $F$,
whereas later sections depend crucially on using the full $F$.

\begin{proposition}
\label{thm:effresvertexmeasure}
For every finite, connected graph $G$ with $w(x, y) \ge 1$ for all edges $(x,
y)$, we have for every vertex $x\in V$ and $\delta\in [\lambda_2, 2)$,
$$\muz_x(\delta) + \pi(x)
\leq \rdiam(x)\delta\weight(x) \leq (n-1)\delta\weight(x)\,.
$$
Therefore for $\delta \in [\lambda_2, 2)$,
$$\muz(\delta) + 1/n
\leq \rdiam \bar w\delta\leq (n-1)\bar w\delta
$$
and 
$$
\lambda_k
\ge
\frac{k}{\tcommute^*}
\ge
\frac{k}{(n-1) \vol (V)}
\,,
$$
where $\bar w := \vol(V)/n$.
\end{proposition}
\begin{proof}
Suppose first that $\muz_x(\delta) > 0$.
Recall that
$F_x:= \Iz(\delta) \b1_x/\weight(x)$.
Define $f$ as in \autoref{cl:fprops}.
By (i) and (iii) of \autoref{cl:fprops}
and \autoref{lem:rayleighquotient}, we have for each
$y \in V$,
$$
\delta \geq \langle \cL f,f\rangle_\weight
= \cE_f(E) \geq |f(x)-f(y)|^2/\rdiam(x) 
\,.
$$
Therefore, 
\begin{align*} 
\delta \rdiam(x)
&\ge
\sum_y w(y) |f(x) - f(y)|^2 /\vol(V)
\\ &=
f(x)^2 + \frac{1}{\vol(V)} - \frac{2}{\vol(V)} f(x) \sum_y w(y)
f(y)
=
f(x)^2 + \frac{1}{\vol(V)} 
\end{align*} 
since $f \perp \b1$.
Use of \autoref{cl:fprops}(ii) now gives the first inequality,
$\muz_x(\delta) + \pi(x) \leq \rdiam(x)\delta\weight(x)$.

Suppose next that $\muz_x(\delta) = 0$. Then to complete the proof of the
first inequality, we must show that $1 \le \rdiam(x) \lambda_2 \vol(V)$.
This is proved by using a unit-norm $\lambda_2$-eigenvector, $f$, of $\cL$.
Then the preceding calculation gives the desired inequality.

Furthermore, since $\weight(y,z)\geq 1$ for all adjacent pairs of vertices,
the conductance of each edge is at least 1. Therefore, since $G$ is
connected, the effective resistance of each pair of vertices is at most
$n-1$ (as in \autoref{lem:pathenergy}).
Hence, $\rdiam(x)\leq n-1$. This completes the proof of
\autoref{thm:effresvertexmeasure}, where for the lower bound on $\lambda_k$
we use \autoref{lem:lambdamutrans}.
\end{proof}

It is known that for lazy random walk,
the $L^\infty$-mixing time is bounded by the maximum
hitting time (see the middle display on p.~137 of \cite{LPW06}), which, in
turn, is at most the maximum commute time. More precisely (recall \eqref{eq:mixingreturnprob}), \cite{LPW06} shows
that 
$$
\frac{p_{t}(x,x)}{\pi(x)}-1 
\leq  
\frac{\sum_y \pi(y) \E_y[T_x]}{t} 
\,,
$$
where $T_x$ is the first time the lazy random walk visits $x$.
This result is due to Aldous.
We give another proof here that the $L^\infty$-mixing time is bounded by
the commute time, which we use to answer open questions on the smallest
log-Sobolev and entropy constants.
\begin{corollary}
\label{cor:inftymixbound}
For every unweighted, loopless, finite, connected graph $G$, lazy simple
random walk satisfies
$$ \tau_\infty(1/4)\leq \big\lceil 16|E|\rdiam \big\rceil\leq 8n^3\,.$$
More generally, for any lazy random walk on a finite, connected, weighted
graph, $G$, 
all $x \in V(G)$, and all $t \ge 1$,
\begin{equation}
\label{eq:rtnbound}
\frac{p_{t}(x,x)}{\pi(x)}-1 
<
\frac{\tcommute^x}{t} 
\,,
\end{equation}
whence 
$$ \tau_\infty(1/4)\leq \big\lceil 4 \tcommute^* \big\rceil\,.$$
\end{corollary}
\begin{proof}
By \autoref{lem:returnprobspectral} and \autoref{thm:effresvertexmeasure}, we have
\begin{align*} 
\frac{p_{t}(x,x)}{\pi(x)}-1 
&\leq  
\frac{1}{\pi(x)} \int_{0}^1 (1-\lambda)^{t}
(\rdiam(x)\weight(x)\lambda)'\,d\lambda
=
\frac{\rdiam(x)\weight(x)}{\pi(x)} \int_0^1 (1-\lambda)^{t}  \,d\lambda \\
&= 
\tcommute^x \int_0^1 (1-\lambda)^{t} \,d\lambda 
<
\tcommute^x \int_0^\infty e^{-\lambda t} \,d\lambda 
=
\frac{\tcommute^x}{t} 
\,.
\end{align*}
If $t := \big\lceil 4 \tcommute^* \big\rceil$, then this is at most $1/4$,
whence 
$\tau_\infty(1/4)\leq t$
by \eqref{eq:mixingreturnprob}.
In the unweighted case, we use the fact that $\vol(V) = 4|E|$ after
loops are added.
\end{proof}

We remark that one may obtain
a somewhat better bound by integrating only from $\lambda_2$, which then 
allows one to reduce $8n^3$ above to $7n^3$.

As is well known, the barbell graph has $\Omega(n^3)$ $L^1$-mixing time.
(This follows from the bound on $\lambda_2$ of \cite{LO81} and, say,
\cite[Theorem 12.4]{LPW06}.)

The 5th open question in \cite{MT06} asks how small the log-Sobolev and
entropy
constants can be for an $n$-vertex unweighted connected graph.
We can now answer this (up to constant factors). We first recall the
definitions.
Define $\Ent_\pi(f) := \bigl\langle f, \log (f/\langle f, \pi \rangle)
\bigr\rangle_\pi$. The \dfn{entropy constant} is 
$$
\rho_0(G) := \inf_f \frac{\langle \cL f, \log f \rangle}{\Ent_\pi f}
\,,
$$
where the infimum is over $f \colon V \to (0, \infty)$ with $\Ent_\pi f
\ne 0$.
The \dfn{log-Sobolev constant} is 
$$
\rho(G) := \inf_f \frac{\langle \cL f, f \rangle}{\Ent_\pi (f^2)}
\,,
$$
where the infimum is over $f \colon V \to \R$ with $\Ent_\pi (f^2) \ne 0$.
It is known \cite[Proposition 2.10]{MT06} that 
$$
4 \rho \le \rho_0 \le 2 \lambda_2
$$
and \cite[Theorem 5.13]{MT06} that 
$$
2 \rho \ge 1/\tau_2(1/e)
\,.
$$
In the latter case, continuous-time random walk is used.
As noted in \cite{MT06}, the first of these inequalities implies that
$\min_G \rho(G) = O(n^{-3})$ and $\min_G \rho_0(G) = O(n^{-3})$ because of the
example of the barbell graph cited earlier, where the minima are over
$n$-vertex unweighted graphs.
On the other side, the continuous-time analogue of
\eqref{eq:rtnbound}, namely, 
$$
\frac{q_t(x, x)}{\pi(x)} - 1
\le
\frac{\tcommute^*}{t}
\,,
$$
yields that $\tau_2(1/e) < e^2 n^3/2$, which combined with the second inequality
above gives $\rho
> 1/(e^2 n^3)$ and $\rho_0 > 1/(e^2 n^3)$. Thus, we have proved the
following:

\begin{theorem}
\label{thm:SobEntConsts}
For finite, unweighted graphs $G$ with $n$ vertices, we have
\begin{equation*}
\min_G \rho(G) = \Theta(n^{-3})
\quad \mbox{and} \quad
\min_G \rho_0(G) = \Theta(n^{-3})
\,.
\end{equation*}
\end{theorem}

For {\em regular} unweighted graphs, we may reduce the mixing bound
$O(n^3)$ of \autoref{cor:inftymixbound} to $O(n^2)$.
To see this,
we use the following well-known bound on growth of regular graphs.
Bounds on the diameter of regular graphs go back to \cite{Moon},
but he uses a different approach.

\begin{lemma}
\label{lem:pathlength}
For every unweighted, connected, $d$-regular graph $G$, $\>x\in V$ and $1\leq
r\leq \diam(x)$, we have $\vol(x,r)\geq d^2r/3$.
In particular, $\diam(x) \le 3n/d$.
\end{lemma}
\begin{proof}
Let $B:=B_{\dist}(x,r)$. Choose $y\in B$ such that $\dist(x,y)=r$.  
Let $\cP:=(y_0,y_1,\ldots,y_{r})$ be a shortest path from $x$ to $y$. 
Let $S:=\{y_0, y_3, y_6,\ldots, y_{3\lfloor (r-1)/3\rfloor}\}$. Since
$\cP$ is a shortest path from $x$ to $y$, no vertex of $S$ is 
adjacent to any other vertex of $S$,  and no pair of vertices of $S$ 
have any common neighbors. Moreover, since for each $z\in S$,
$\dist(x,z)<r$, each vertex of $S$ is adjacent only to the vertices inside
$B$. Therefore, since $G$ is $d$-regular, every vertex of $S$ has $d-2$
unique neighbors in $B\setminus \cP$ that are not adjacent to any other
vertices of $S$. Hence
$$ |B| \geq |\cP| + |S|(d-2) \geq (r+ 1) + \frac{(d-2)r}{3} \geq
\frac{(d+1)\cdot r}{3}\,.$$
Since $G$ is $d$-regular, we get $\vol(B) = \vol(x,r) \geq d^2r/3$.
\end{proof}

\begin{corollary}
\label{cor:regularmix}
For every unweighted, finite, connected, regular graph $G$, we have
$$ 
\tau_\infty(1/4) \leq 24 n^2
\,.
$$
\end{corollary}
\begin{proof}
Let $d$ be the degree of $G$.
Since $|E| = nd/2$ and $\rdiam \le \diam \le 3n/d$, the inequality is
immediate from \eqref{eq:rtnbound}.
\end{proof}

As is well known \cite[Example 3.11]{MT06}, $\tau_\infty(1/4) =
\Theta(n^2)$ for a cycle on $n$ vertices.

We remark that the same bound as in \autoref{cor:regularmix}
holds with an extra factor of the maximum
degree over the minimum degree for general finite, unweighted graphs.

\subsection{Volume-Growth Conditions}

We now prove stronger bounds that depend on lower bounds for volume growth.
Our first proof has some similarity with that of \cite[Lemma 2.4]{BCG}.

\begin{proposition}
\label{thm:generalvertexmeasure}
Let $G$ be a finite or infinite graph that satisfies
$w(x, y) \ge 1$ for all edges $(x,y)$. Then
for every vertex $x\in V$, $\>\delta\in(0,2)$, and $\alpha \in (0, 1)$,
\begin{equation}
\label{eq:generalvertexmeasure}
\muz_x(\delta) \leq \frac{\delta \weight(x)}{\alpha^2} r
\quad\mbox{when}\quad
\vol(x,r) > \frac{\weight(x)}{(1-\alpha)^2\muz_x(\delta)}
\,.
\end{equation}
Thus, 
\begin{equation}
\label{eq:generalvertexmeasure2}
\muz_x(\delta) \leq \frac{4 \weight(x)}{\vol(x, r)}
\quad\mbox{for}\quad \delta \le \frac{1}{r \vol(x, r)}
\,.
\end{equation}
\end{proposition}
\begin{proof}
We may assume that $\muz_x(\delta) > 0$.
Let $f$ be as defined in \autoref{cl:fprops}, and let
$B:=B_f\bigl(x,\alpha f(x)\big)$. 
If $G$ is finite, then $f$ is \spheresym, so there exists a vertex outside
of $B$ by \autoref{lem:spheresymball}.
If $G$ is infinite, then there exists a vertex outside of $B$ by
\autoref{cl:fprops}(i). 
Let $\cP$ be
a shortest path from $x$ to a vertex outside of $B$ (since $G$ is
connected, some such $\cP$ exists). Since
$f\in \img\bigl(\Iz(\delta)\big)$ by \autoref{cl:fprops}, we have by
\autoref{lem:rayleighquotient} that
\begin{equation}
\label{eq:rayleighaverageupper}
 \delta \geq \langle \cL f, f\rangle_\weight = \cE_f(E) \geq \cE_f(B) \geq
 \frac{\alpha^2 f^2(x)}{|\cP|} = \frac{\alpha^2
 \muz_x(\delta)}{\weight(x)|\cP|}\,,
 \end{equation}
where the third inequality holds by \autoref{lem:pathenergy}. 
Let $B':=B_{\dist}\bigl(x,|\cP|-1\big)$.
By definition of $\cP$, we have $B'\subseteq B$. 
Since
$$ \vol(B)(1-\alpha)^2 f^2(x) \leq \sum_{y\in B} |f(y)|^2 \weight(y) \leq
\sum_{y\in V} |f(y)|^2 \weight(y) =\norm{f}^2_\weight=1\,, $$
we obtain
\begin{equation}
\label{eq:volpathupper}
 \vol(x,|\cP|-1)=\vol(B') \leq \vol(B) \leq \frac{1}{(1-\alpha)^2f^2(x)} =
 \frac{\weight(x)}{(1-\alpha)^2\muz_x(\delta)}\,.
 \end{equation}
This means that in \eqref{eq:generalvertexmeasure},
we have $r \ge |\cP|$. Therefore, \eqref{eq:generalvertexmeasure}
follows from \eqref{eq:rayleighaverageupper}.

If we combine the two inequalities 
$\vol(x,r) > \frac{\weight(x)}{(1-\alpha)^2\muz_x(\delta)}$ and
$\muz_x(\delta) \leq \frac{\delta \weight(x)}{\alpha^2} r$ for $\alpha =
1/2$, then we obtain
that the first of them implies that $r\delta > 1/\vol(x, r)$. The
contrapositive of this is \eqref{eq:generalvertexmeasure2}.
\end{proof}

For infinite graphs, this result can be compared to \cite[Theorem
21.18]{LPW06} (due to \cite[Proposition 3.3]{BCK}), a version of
which can be stated as 
\begin{equation}
\label{eq:BCK}
p_t(x, x)
\le
\frac{3 w(x)}{\vol(x, r)}
\end{equation}
for $t \ge r \cdot \vol(x, r)$, provided the random walk is lazy.
This result implies \eqref{eq:generalvertexmeasure2} with ``4" replaced by
``$6e$" via \autoref{lem:reversereturnprobspectral}.

Next we describe some of the straightforward corollaries of the preceding
proposition:
\begin{corollary}
For every finite or infinite, connected graph $G$ with $w(x, y) \ge 1$ for all
edges $(x, y)$ and every $x\in V$ and $\delta\in (0,2)$,
$$ \muz_x(\delta) \leq  \max\big\{\weight(x)\sqrt{12\delta},
2w(x)/\diam(x)\big\}\,.$$
\end{corollary} 
\begin{proof}
Since $\weight(y,z)\geq 1$ for all adjacent pairs of vertices, for any
simple path $\cP$ of length
$r$, we have $\vol(\cP)\geq 2r$. Thus, $\vol(x,r)\geq 2r$. Therefore, by
\autoref{thm:generalvertexmeasure}, for $\alpha=1/2$ and $r =
\lceil\frac{2\weight(x)}{\muz_x(\delta)}\rceil$, we get, provided that $r
\le \diam(x)$,
$$ \muz_x(\delta) \leq 4\delta\weight(x)r\leq
4\delta\weight(x)\left(
\frac{2\weight(x)}{\muz_x(\delta)}+1\right) \leq
\frac{12\delta\weight^2(x)}{\muz_x(\delta)}\,,$$
where the last inequality holds by the fact that $\weight(x)\geq 1$ and
$\muz_x(\delta)\leq 1$. 
If, on the other hand, $r \ge \diam(x) + 1$, then 
$\frac{2\weight(x)}{\muz_x(\delta)} \ge \diam(x)$, which completes the proof.
\end{proof}

For regular unweighted graphs, we can remove the dependence above on $w(x)$.
It appears that this result is new.

\begin{theorem}
\label{prop:regularreturn}
For every unweighted, loopless,
connected, regular graph $G$ and every $x\in V$, we
have $\muz_x(\delta) < 10\sqrt{\delta}$. 
Hence if $G$ is finite, $\muz(\delta) < 10\sqrt{\delta}$ and
for $2\leq k\leq n$, we have
$$\lambda_k > \frac{(k-1)^2}{100n^2}\,.$$
For all $t > 0$ and $x \in V$, lazy simple random walk satisfies
$$
p_t(x, x) - \pi(x)
<
\frac{13}{\sqrt t}
\,.
$$
\end{theorem}

\begin{proof}
Let $f$ be as defined in \autoref{cl:fprops}, let
$B:=B_f\bigl(x,\alpha f(x)\big)$ for $\alpha=1/2$.
Since $f \perp \b1$, we have
$B \ne V$, and thus we may choose 
a shortest path $\cP$ from $x$ to the outside of $B$.
Write $d$ for the degrees of the vertices of $G$.
We want to show that $\vol(B) \ge d^2|\cP|/6$, and then the proof
that $\muz_x(\delta) < 10 \sqrt \delta$
follows by equations \eqref{eq:rayleighaverageupper} and
\eqref{eq:volpathupper}.
Unfortunately, this lower bound on $\vol(B)$ may not hold in the case $|\cP|=1$. Suppose that $|\cP|=1$ and $\vol(B)< d^2/2$. Then it must be that at least half of the neighbors of $x$ are outside of $B$. Therefore,
$$ \delta \geq \langle Lf,f\rangle\geq  \cE_f(B) \geq \frac{\alpha^2 f^2(x)
d}{2} = \frac{ \muz_x(\delta)}{8}\,,$$
and we are done. 

So, if $|\cP|=1$ we may assume that $\vol(B)\geq d^2|\cP|/2$. 

If $|\cP|\geq 2$, then by \autoref{lem:pathlength},
$$\vol(B) \geq \vol(x,|\cP|-1)\geq d^2(|\cP|-1)/3 \geq d^2|\cP|/6\,.$$
Thus, we may assume the above equation holds for all $|\cP| \ge 1$.
Substituting this into \eqref{eq:volpathupper} yields
$|\cP|\leq \frac{24}{d\muz_x(\delta)}$.
Finally, by \eqref{eq:rayleighaverageupper} we obtain
$$ \delta \geq \frac{\alpha^2 \muz_x(\delta)}{d|\cP|} \geq
\frac{\alpha^2\muz_x(\delta)^2}{24} >  \frac{\muz_x(\delta)^2}{100}\,. $$
Since the above equation holds for every vertex $x\in V$, it holds also for
the spectral measure of $G$ as well, i.e., 
$\muz(\delta) < 10 \sqrt \delta$.
The inequality on $\lambda_k$ then
follows by an application of \autoref{lem:lambdamutrans}.

Lastly, the bound on return probabilities follows from
\autoref{lem:returnprobspectral}:
since the spectral measure for lazy simple random walk
satisfies $\muz_x(\delta) < 10 \sqrt{2\delta}$, we have
\[
p_t(x, x) - \pi(x)
<
\int_0^1 e^{-\lambda t}\frac{10}{\sqrt{2\lambda}} \,d\lambda
<
\frac{10}{\sqrt {2t}} \int_0^\infty e^{-s} s^{-1/2} \,ds
=
\frac{10\sqrt{\pi}}{\sqrt {2t}}
<
\frac{13}{\sqrt t}
\,.
\qedhere
\]
\end{proof}

Again, we remark that the same upper bounds hold with an extra factor of
the maximum degree over the minimum degree for general unweighted graphs
(or the square of the reciprocal of this factor for the lower bound on
$\lambda_k$).

Of course, the example of cycles shows that the bounds are sharp up to
constants.

We may also illustrate \autoref{thm:generalvertexmeasure} by choosing
common growth rates, as in the following two corollaries.
The bound on return probabilities in the first corollary is the same as
\cite[Example 2.1]{BCG}, except for the constant, which was left implicit
in \cite{BCG}. (Note that all their results on graphs, including
Theorem 2.1, require the hypothesis that $\weight(x)$ be uniformly bounded.
This was assumed in \cite[Proposition V.1]{Coulhon:ultra} that they used.)
The result is sharp up to a constant factor for every growth rate
$D$ in the first corollary, even for
unweighted graphs with bounded degree, as shown by \cite[Theorem 5.1]{BCG}
in combination with \autoref{lem:bothreturnprobspectral}: choose $t :=
c/\delta$ for a sufficiently large constant $c$.

\begin{corollary}
\label{thm:polyvertexmeasure}
Let $G$ be an infinite graph with $w(x, y) \ge 1$ for all edges $(x,y)$ and
$x\in V$. 
Suppose that $c > 0$ and $D \ge 1$ are
constants such that for all $r \ge 0$, we have $\vol(x, r) \ge c (r+1)^D$. Then
for all $\delta\in(0,2)$, 
\begin{equation*}
\label{eq:polyvertexmeasure}
\muz_x(\delta) \leq C \weight(x) \delta^{D/(D+1)}
\,,
\end{equation*}
where 
$$
C := \frac{(D+1)^2}{c^{1/(D+1)} D^{2D/(D+1)}}
\,.
$$
Hence for all $t \ge 1$, lazy simple random walk satisfies
$$
p_t(x, x) 
<
C' w(x) t^{-D/(D+1)}
\,,
$$
where
$$
C' := \frac{2^{D/(D+1)} (D+1)}{c^{1/(D+1)} D^{(D-1)/(D+1)}} \Gamma\Big(\frac{D}{D+1}\Big)
\,.
$$
\end{corollary}

For this corollary,
recall the definition of the
gamma function, $\Gamma(z) := \int_0^\infty e^{-t} t^{z-1} \,dt$.
See the appendix for a proof of the corollary. 

For example, we may always take $c = D = 1$, in which case we obtain the
bounds $\muz_x(\delta) \le 4w(x)\sqrt\delta$ and $p_t(x, x) <
2\sqrt{2\pi} w(x)/\sqrt t$. For comparison, \cite[Lemma 3.4]{Lyo05} gives
the slightly better bound $p_t(x, x) \le 2w(x)/\sqrt{t+1}$.

Similarly, one can prove the following:

\begin{corollary}
\label{thm:superpolyvertexmeasure}
Let $G$ be an infinite graph with $w(x, y) \ge 1$ for all edges $(x,y)$ and
$x\in V$. 
Suppose that $c_1, c_2, a > 0$ are
constants such that for all $r \ge 1$, we have $\vol(x, r) \ge c_1 e^{c_2
r^a}$. Then
for all $\delta\in\bigl(0,\min\{2, c_2^{1/a}/(2c_1 e)\}\big)$, 
\begin{equation*}
\label{eq:superpolyvertexmeasure}
\muz_x(\delta) \leq 8 c_2^{-1/a} \weight(x) \delta \left(\ln
\frac{c_2^{1/a}}{2 c_1 \delta} \right)^{1/a}
\,.
\end{equation*}
\end{corollary}

One can deduce from this or, more directly, from \eqref{eq:BCK} that under
the same hypotheses, 
\[
p_t(x, x) 
\le
C w(x) \frac{(\log t)^{1/\alpha}}t
\]
for some constant $C = C(c_1, c_2, a)$.
It is open whether
these inequalities are sharp up to the dependence on constants;
\cite[Example 5.2]{BCG} shows
the existence of $G$ that satisfies
\[
\sup_{x \in V} p_t(x, x) 
\ge
C w(x) \frac{(\log t)^{1/\alpha - 1}}t
\,.
\]

\section{Bounds on Average Spectral Measure} \label{sec:loweravg}

In this section, we consider only finite graphs.
In the preceding section, we proved a sharp $O(\delta)$ bound on $\muz(\delta)$,
but the implicit constant depended on the graph
(\autoref{thm:effresvertexmeasure}).
In the regular unweighted case, we
obtained a sharp $O(\!\sqrt\delta\,)$ bound with a universal constant
(\autoref{prop:regularreturn}). Here, we
obtain a sharp $O(\delta^{1/3})$ bound with a universal constant for all
unweighted graphs. 
This answers a question of \cite{Lyo05} (see (3.14) there) and has an
application to estimating the number of spanning trees of finite graphs from
information on neighborhood statistics; see below.
No such bound on $\muz_x(\delta)$ for individual
vertices $x$ is valid, however.

\begin{theorem}
\label{thm:generalgraphseig}
For every finite, unweighted, loopless, connected graph $G$, and every
$\delta\in(0,2)$, we have $\muz(\delta) < (4000\delta)^{1/3}$ and for $2
\le k \le n$, we have
$$\lambda_k > \frac{(k-1)^3}{4000 n^3}\,.$$
\end{theorem}

For each $k$,
this is sharp up to a constant factor, as shown by the following example:
We may assume that $k < n/6$. 
Let $G$ consist of $k$ cliques of size $\sim 2n/(3k)$ joined in a cycle by
paths of length $\sim n/(3k)$
(see \autoref{fig:dumbbellgraph} for an illustration).
For each $i = 1, \ldots, k$, define $f_i$ to be the function that is 1 on the
$i$th clique and goes to 0 linearly on each of the paths leaving that
clique, reaching 0 at the midpoint and having value 0 elsewhere.
It is straightforward to calculate that $\cR(f_i) \sim 27k^3/n^3$.
Since the supports of all $f_i$ are pairwise separated, i.e., no vertex in
the support of $f_i$ is adjacent to any vertex in the support of $f_j$ for
$i \ne j$, the same asymptotic 
holds simultaneously for the Rayleigh quotient of every function $\ne \b0$
in the linear span of the $f_i$, whence $\lambda_k \le \bigl(27 +
o(1)\big)k^3/n^3$.

\begin{SCfigure}[.85][ht]
\begin{tikzpicture}[inner sep=1.8pt,scale=.8,pre/.style={<-,shorten <=2pt,>=stealth,thick}, post/.style={->,shorten >=1pt,>=stealth,thick}]

\tikzstyle{every node} = [draw, circle,fill,black];
\foreach \a/\l in {45/a, 135/b, 225/c, 315/d}{
\foreach \b in {0, 60, 120, 180, 240, 300}{
\path (\a:2.8)+(\b:.8) node [fill=red] (\l_\b){};
}
\foreach \b/\c in {0/60, 0/120, 0/180, 0/240, 0/300, 60/120, 60/180, 60/240, 60/300, 120/180,
120/240, 120/300, 180/240, 180/300, 240/300}{
\path  (\l_\b) edge  (\l_\c);
}
}
\foreach \x/\y/\l in {2.4/-0.6/1, 2.4/0/2, 2.4/0.6/3, 0.6/2/4, 0/2/5, -0.6/2/6,  -2.4/0.6/7, -2.4/0/8, -2.4/-0.6/9,  -0.6/-2/10, 0/-2/11, 0.6/-2/12}{
\path (0:0)+(\x,\y) node [fill=red] (e_\l){};
}
\path (a_180) edge (e_4); \path (e_4) edge (e_5); \path (e_5) edge (e_6);  \path (e_6) edge (b_0);
\path (b_240) edge (e_7); \path (e_7) edge (e_8); \path (e_8) edge (e_9); \path (e_9) edge (c_120);
\path (c_0) edge (e_10); \path (e_10) edge (e_11); \path (e_11) edge (e_12); \path (e_12) edge (d_180); 
\path (d_60) edge (e_1); \path (e_1) edge (e_2); \path (e_2) edge (e_3); \path (e_3) edge (a_300);

\end{tikzpicture}
\hspace{.1\textwidth}

\caption{An example of a graph where $\lambda_k=\Omega(k^3/n^3)$. In this
graph, each clique has size $\Theta(n/k)$ and cliques are connected by
paths of length $\Theta(n/k)$. 
\vspace{30pt}}
\label{fig:dumbbellgraph}
\end{SCfigure}
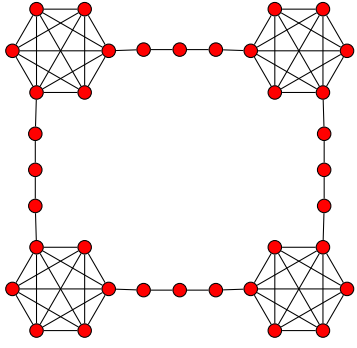

We prove the above theorem by showing that $\cR(F) =
\Omega\bigl(\mu(\delta)^3\big)$. %
Our proof is a generalization of the proof of
\autoref{thm:generalvertexmeasure}.  Here,  instead of just lower-bounding
the Rayleigh quotient by considering a ball around a single vertex, we
take $\Omega(k)$ disjoint balls about $\Omega(k)$ vertices chosen
carefully so that their spectral measure is within a constant factor of
the average. 
This requires us to use the higher-dimensional embedding $F$, not merely
its 1-dimensional relative $f$.

Let $m:=\lfloor \muz(\delta)n/2\rfloor + 1$. We use 
Algorithm~\ref{alg:ballselection}
to choose $m$ disjoint balls  based on the spectral embedding of
$G$.

\begin{algorithm}
\begin{algorithmic}
\STATE Let $S_0\leftarrow V$.
\FOR {$i=1 \to m$}
\STATE Choose a vertex $x_i$ in $S_{i-1}$ that maximizes $\muz_{x_i}(\delta)$.
\STATE Let $S_{i}\leftarrow S_{i-1}\setminus B_F\bigl(x_i, \alpha
\norm{F_{x_i}}_\weight\big)$.
\ENDFOR
\RETURN $B_F\bigl(x_1, \alpha\norm{F_{x_1}}_\weight\big), \ldots, B_F\bigl(x_{m},
\alpha\norm{F_{x_{m}}}_\weight\big)$.
\end{algorithmic}
\caption{Ball-Selection($\alpha$)}
\label{alg:ballselection}
\end{algorithm}

The next lemma shows properties of Ball-Selection that will be used in the
proof. In the rest of the proof, we let $\alpha:=1/4$.
\begin{lemma}
\label{lem:ballprops}
The returned balls satisfy
\begin{enumerate}[\rm i)]
\item for each $1\leq i\leq m$, we have
$\muz_{x_i}(\delta) \geq \muz(\delta)/3$
and
\item for every $1\leq i < j \leq m$,
$$ B_F\Bigl(x_i, \frac{1}{9} \norm{F_{x_i}}_\weight\Big) \cap B_F\Bigl(x_j,
\frac{1}{9}\norm{F_{x_j}}_\weight\Big) = \varnothing\,.$$
\end{enumerate}
\end{lemma}
\begin{proof}
First observe that by property (ii) of \autoref{lem:isotropy}, for each $1\leq i\leq m$, we have
 $$\muz_{B_F(x_i,\alpha\lVert{F_{x_i}}\rVert_\weight)} (\delta) \leq \frac{1}{1-4\alpha^2}
 = 4/3\,.$$

Since  $\muz_{S_0}(\delta)=\muz_{V}(\delta)=n\muz(\delta)$, and, by the
above equation, the spectral measure of the removed vertices in each
iteration of the for loop is at most 4/3, we obtain
$$ \muz_{S_{m-1}}(\delta) \geq
n\muz(\delta) - (m-1) 4/3 \geq n\muz(\delta)/3\,,$$
where last inequality holds by the definition of $m$.
Since $x_i$ has the largest spectral measure in $S_{i-1}$ for $1\leq i\leq m$, we have
$$ \muz_{x_i}(\delta) \geq \muz_{S_{i-1}}(\delta)/n \geq
\muz_{S_{m-1}}(\delta)/n \geq \muz(\delta)/3\,.$$
This proves (i).

To see (ii), suppose that some $y$ lies in both balls: $\norm{F_y -
F_{x_\ell}}_\weight \le \norm{F_{x_\ell}}_\weight/9$ 
\vadjust{\kern2pt}%
for both $\ell = i, j$.
If $\norm{F_{x_j}}_\weight \le (5/4)\norm{F_{x_i}}_\weight$, 
then the triangle inequality gives $\norm{F_{x_i} - F_{x_j}}_\weight \le
\alpha \norm{F_{x_i}}_\weight$,
i.e., $x_j \in B_F\bigl(x_i, \alpha\norm{F_{x_i}}_\weight\big)$,
contradicting $x_j \in S_i$. On the
other hand, if $\norm{F_{x_j}}_\weight > (5/4)\norm{F_{x_i}}_\weight$, then
$\norm{F_y}_\weight \ge (8/9) \norm{F_{x_j}}_\weight > (10/9)
\norm{F_{x_i}}_\weight$, which contradicts $y \in B_F\bigl(x_i,
\norm{F_{x_i}}_\weight/9\big)$.
\end{proof}

In the rest of the proof, let $B_i := B_F\bigl(x_i, 
\norm{F_{x_i}}_\weight/9\big)$ for all $1\leq i\leq m$. 
In the next lemma, we prove strong lower bounds on the energy of every
ball $B_i$. Then we will bound the numerator of the Rayleigh
quotient of $F$ from below simply by adding up these lower bounds.

\begin{lemma}
\label{lem:energy} For every $1\leq i\leq m$,
$$ \cE_F(B_i) > \frac{\muz(\delta)}{250\,  |B_i|^2}\,. $$
\end{lemma}
\begin{proof}
We consider two cases. If $\weight(x_i)\leq |B_i|$, then we lower-bound
$\cE_F(B_i)$ by measuring the energy of the edges of a shortest path
from $x_i$ to the outside. Otherwise, we simply lower-bound
$\cE_F(B_i)$ by the stretch of edges of $x_i$ to its neighbors outside
of $B_i$.

Since $F$ is a \spheresym~embedding by \autoref{lem:spheresym},  there is a
vertex outside of each ball $B_i$ by \autoref{lem:spheresymball}.
Let $\cP_i$ be a shortest path (with respect to the graph distance in $G$)
from $x_i$ to any vertex outside of $B_i$.  Since $G$ is connected, some
such $\cP_i$ exists.
Using \autoref{lem:pathenergy}, we can lower-bound the
energy of $B_i$ by
\begin{equation}
\label{eq:ballbound}
\cE_F(B_i) \geq \frac{\norm{F_{x_i}}_\weight^2}{81\,|B_i|} =
\frac{\muz_{x_i}(\delta)}{81 \cdot \weight(x_i) \cdot |B_i|} 
>
\frac{\muz(\delta)}{250\cdot \weight(x_i)\cdot |B_i|}\,,
\end{equation}
where the equality holds by \autoref{lem:normF} and the second inequality
holds by (i) of \autoref{lem:ballprops}.
By the above inequality, if $\weight(x_i) \leq |B_i|$, then
$\cE_F(B_i) > \frac{\muz(\delta)}{250\, |B_i|^2}$, and we are done. 

On the other hand, suppose that
$\weight(x_i) > |B_i|$.
Let $\nei(x)$ denote the set of neighbors of $x$ in $G$. 
Since $G$ is a simple graph, at least $\weight(x_i) - |B_i|+1$ of the
neighbors of $x_i$ in $G$ are not contained in $B_i$. That is, $|\nei(x_i)
\setminus B_i | \geq \weight(x_i)-|B_i|+1$. We lower-bound the energy of
$B_i$ by the energy of the edges between $x_i$ and its neighbors that are not
contained in $B_i$:
\begin{align*}
\cE_F(B_i) \geq   \sum_{\substack{y\sim x_i\\ y\notin B_i}}
\norm{F_{x_i} - F_y}_\weight^2 &\geq    |\nei(x_i) \setminus B_i|
\frac{1}{81} \norm{F_{x_i}}_\weight^2  \\
&>    \bigl(\weight(x_i)- |B_i|+1\big)
\frac{\muz(\delta)}{250\cdot \weight(x_i)} > \frac{\muz(\delta)}{250\,|B_i|}
\,.
\end{align*}
The second inequality 
uses the radius of the ball $B_i$, the third
inequality follows as in \eqref{eq:ballbound}, and
the last inequality follows by the assumption
$\weight(x_i) > |B_i|$.
\end{proof}

Now we are ready to lower-bound $\cR(F)$.

\begin{proofof}{\autoref{thm:generalgraphseig}}
We may assume that $\muz(\delta) > 0$, i.e., $\delta \ge \lambda_2$, since otherwise the inequality is
trivial.
By property (ii) of \autoref{lem:ballprops}, the balls are disjoint.
Therefore, $\sum_{i=1}^{m} |B_i| \leq n$. Hence
\autoref{lem:rayleighquotientF} yields
\begin{align*}
\delta 
&\geq \cR(F) =\frac{\sum_{x\sim y} \norm{F_x-F_y}_\weight^2}{\sum_{y}
\norm{F_y}_\weight^2 \weight(y)}
\geq \frac{1}{2n\muz(\delta)} \sum_{i=1}^{m} \cE_F(B_i) 
\\ &>
\frac{1}{2n\muz(\delta)} \sum_{i=1}^{m} \frac{\muz(\delta)}{250\,  |B_i|^2} \geq
\frac{m^3}{500 n^3}\geq\frac{\muz(\delta)^3}{4000}\,,
\end{align*}
where the second inequality follows by \autoref{lem:normF} and the fact 
that each edge is counted in at most two balls, the fourth inequality
follows by convexity of the function $s \mapsto 1/s^2$,
and the last inequality holds by the fact
that $m\geq n\muz(\delta)/2$. This completes the proof of
\autoref{thm:generalgraphseig}.
\end{proofof}

As a corollary of the above theorem, we  can upper-bound the average return
probability of lazy simple random walk  on every finite connected graph.
This can also be expressed in terms of a version of mixing, since
\[
\sum_{x\in V} p_{2t}(x,x) - 1
=
\sum_{x \in V} \pi(x) \sum_{y \in V} \left|\frac{p_t(x, y)}{\pi(y)} -
1\right|^2 \pi(y)
\,.
\]
\begin{corollary}
\label{thm:universalavg}
For every unweighted, finite, connected graph $G$ and every integer $t\geq
1$, lazy simple random walk satisfies
$$
\frac{\sum_{x\in V} p_t(x,x) - 1}{n} < \frac{18}{t^{1/3}}
\,.
$$
\end{corollary}
\begin{proof}
By \autoref{lem:returnprobspectral} and \autoref{thm:generalgraphseig}, we
may write
\begin{align*}
\frac{1}{n}\Big(\sum_{x\in V}p_t(x,x) -1\Big) &< \int_0^2
e^{-\lambda t} \bigl((4000\cdot2\lambda)^{1/3}\bigr)'\, d\lambda 
= \frac{20}{3} \int_0^1 e^{-\lambda t}\lambda^{-2/3}\, d\lambda\\
 &< \frac{20}{3t^{1/3}} \int_0^\infty e^{-s} s^{-2/3}ds <
 \frac{18}{t^{1/3}}\,.
 \qedhere
\end{align*}
\end{proof}

This bound is sharp up to a constant factor as shown by the example of a
barbell graph.

Our interest in this type of inequality is due to its application to
counting the number $\cp(G)$ of spanning trees of large finite graphs $G$.
This relies on \cite[Proposition 3.1]{Lyo05}, which says the following:

\begin{proposition}
\label{prop:expanddet}
Suppose that $\gh$ is a finite, unweighted, loopless, connected graph.
Then 
$$
\log \cp(\gh) 
=
-\log \bigl(4|\edges|\big) + \sum_{x \in \verts} \log
\bigl(2\weight(x)\bigr)
- \sum_{t \ge 1} 
\frac{1}{t} \Big(\sum_{x \in \verts} p_t(x, x) - 1 \Big)
\,,
$$
where $p_t$ refers to lazy simple random walk on $\gh$.
\end{proposition}

For the convenience of the reader, we have reproduced the proof in the
appendix.

As a consequence, we can estimate the
number of spanning trees of simple graphs by knowing only local
information.
For a finite graph $H$ with distinguished vertex $o$, let $p_{r, H}(\gh)$
denote the proportion of vertices $x$ of $\gh$ such that there is an
isomorphism from $B_{\dist}(x,r)$ to $H$ that maps $x$ to $o$.
In \cite{Lyo05}, it is shown that the numbers $p_{r, H}(\gh)$
determine the number $\cp(\gh)$ of spanning trees of $\gh$ by the infinite
series above that converges at a rate determined by the average degree of $\gh$.
In the case of simple graphs, \cite{Lyo05} suggested that a result
like \autoref{thm:universalavg}
would be true, with the result that one has a uniform
approximation to $\log \cp(\gh)$ for simple graphs:

\begin{corollary}
\label{cor:counting}
Given $r \ge 2$,
there is a function of the numbers $p_{r, H}(\gh)$ and $|V(\gh)|$
for (simple, connected) graphs $\gh$ that gives
$|V|^{-1}\log\cp(\gh)$ with an error less than
$45/r^{1/3}$.
In fact, there is such a function that depends only on the map
$(x, t) \mapsto \bigl(w(x), p_t(x, x)\big)$ on $V \times [1, 2r)$,
where $p_t$ refers to lazy simple random walk on $\gh$.
\end{corollary}

\begin{proof}
Fix $r \ge 2$.  
Then
\begin{align*}
\Big|\log \cp(\gh)
&+\log \bigl(4|\edges|\big) - \sum_{x \in \verts} \log
\bigl(2\weight(x)\bigr)
+ \sum_{1 \le t < 2r} 
\frac{1}{t} \Big(\sum_{x \in \verts} p_t(x, x) - 1 \Big) 
\Big|
\\ &=
\sum_{t \ge 2r} 
\frac{1}{t} \Big(\sum_{x \in \verts} p_t(x, x) - 1 \Big) 
<
|\vertex| \sum_{t \ge 2r} \frac{18}{t^{4/3}} 
<
|\vertex| \frac{45}{r^{1/3}}
\,.
\end{align*}
Knowing the ball of radius $r$ about
$x$ determines $p_t(x, x)$ for all $t < 2r$.
Of course, the distribution of the degrees $\weight(x)$ is determined by the
neighborhoods of radius 1 and $2|\edges(\gh)| = \sum_x \weight(x)$.
Thus, the desired function is
\[
|\vertex|^{-1}\left(
-\log \bigl(4|\edges|\big) + \sum_{x \in \verts} \log
\bigl(2\weight(x)\bigr)
- \sum_{t = 1}^{2r-1} 
\frac{1}{t} \Big(\sum_{x \in \verts} p_t(x, x) - 1 \Big)
\right)
\,.
\qedhere
\]
\end{proof}

For the next corollary, we design a local algorithm that approximates the
number of spanning trees of massive graphs. Our algorithm uses only an oracle
satisfying the operations:  select  a uniformly random vertex of $G$,
 select a uniformly random neighbor of a given vertex $x$,
return the degree of a given vertex $x$.
The algorithm  also uses knowledge of $n$ and $|E|$.
For any given $\eps>0$, it approximates
$\frac{1}{n}\log\cp(G)$ within an $\eps$-additive error using only
$O\bigl(\poly(\epsilon^{-1}\log n)\big)$ queries. 

\begin{corollary}
Let $G$ be an unweighted, finite, connected graph. Given an oracle access to $G$ that satisfies the above operations, together
with knowledge of\/ $n$ and $|E|$,
there is a randomized algorithm that for any given $\eps,\delta>0$,
approximates $\log\cp(G) / |V|$ within an additive error of $\eps$, with
probability at least $1-\delta$, by using only
$\tilde{O}(\eps^{-5}+\eps^{-2}\log^2{n})\log \delta^{-1}$ many oracle
queries. 
\end{corollary}
\begin{proof}
Choose $r := \ceil{90^3 \eps^{-3}}$, so that $45 r^{-1/3} \leq \eps/2$. Write
$s:=\sum_{1\leq t< 2r} 1/t$.
Let $W:=\frac{1}{n} \sum_{x} \log \bigl(2\weight(x)\bigr)$ and $Y:=\sum_{x}
\frac{1}{n} \sum_{t=1}^{2r-1} p_t(x,x)/(st)$.
Then by the proof of \autoref{cor:counting}, 
$$ \left|\frac{\log \cp(G)}{n} - \left( \! - \frac{\log(4|E|)}{n} + W - sY +
\frac{s}{n} \right) \right| \leq \eps/2
\,.
$$
Therefore, we just have to approximate $W-sY$ within an additive error of
$\eps/2$. The details of the algorithm are described below.
\begin{algorithm}
\begin{algorithmic}
\STATE Let $r \leftarrow \ceil{90^3 \eps^{-3}}$ and
$s\leftarrow \sum_{1\leq t < 2r} 1/t$.
\STATE $N\leftarrow \ceil{8\log(4/\delta) s^2/\eps^2}$.
\FOR {$i= 1 \to N$}
\STATE Let $x$ be a randomly chosen vertex of $G$.
\STATE Sample $1\leq t< 2r$ with probability $1/st$.
\STATE Run a $t$-step lazy simple random walk from $x$, and let $Y_i
\leftarrow \I{X_t=x}$. 
\ENDFOR
\STATE Sample $\ceil{256\log(1/\delta) (\log{n})^2 / \eps^2}$ random vertices
of $G$, and let $\widetilde{W}$ be the average of the logarithm of twice
\vadjust{\kern2pt}%
the degree of sampled vertices. 
\label{step:avergedeg}
\RETURN $-n^{-1}{\log(4|E|)} + \widetilde{W}- s(Y_1+\ldots+Y_N)/N + s/n$.
\end{algorithmic}
\caption{Approximate Spanning Trees ($\eps$)}
\label{alg:spanningtree}
\end{algorithm}

We start by describing how to approximate $Y$ within an $\eps/4s$ error
(hence, to approximate $sY$ within an $\eps/4$ error). We use a Monte
Carlo sampling method.
Let $X_0,X_1,\ldots,X_t$ represent a $t$-step lazy simple
random walk started from a vertex of $G$. Then
\begin{eqnarray*} Y &=& \sum_{x\in V}\frac{1}{nst}
\sum_{t=1}^{2r-1} \P{X_t=x \mid X_0=x} \\
 &=& \sum_{x\in V}  \sum_{1\leq t<2r}\sum_{\substack {x_1,\ldots,x_t\in V\\
 x_t = x}} \frac{1}{nst} \P{X_1=x_1,\ldots,X_t=x_t\mid X_0=x} 
\,.
\end{eqnarray*}
Consider a random walk starting at a random vertex and lasting a random
length of time. Namely,
let $\cD$ be the distribution on walks of lengths in $[1, 2r)$ where 
$$ \PP{\cD}{(x_0,x_1,\ldots,x_t)} = \frac{1}{nst} \P{X_1=x_1,\ldots,X_t=x_t
\mid X_0=x_0}.$$
Then $Y=\PP{\cD}{\{(x_0,x_1,\ldots,x_t) : x_t=x_0\}}$.
First we describe how to sample from $\cD$, then show how to approximate
$Y$. First we sample a random vertex $x$ of $G$, then we select a random
$1\leq t<2r$ with probability $1/st$ (note that $\sum_{1\leq t< 2r} 1/st=1$
by definition of $s$). Finally, we choose a $t$-step random walk
started from $y$ and compute $\I{X_t=x}$. 
See the details in Algorithm \ref{alg:spanningtree}.

We approximate $Y$ by sampling $N:=\ceil{8\log(4/\delta) s^2/\eps^2}$
independent walks with distribution $\cD$ and  computing their average.
Let $Y_i := \I{X_t=x}$ be computed from
the $i$th sample of $\cD$. By definition,
$Y_i\in \{0,1\}$ and $\E[Y_i]=Y$. Since $Y_1,\ldots,Y_N$ are independent, 
Hoeffding's inequality gives
$$ \P{\Bigl|\frac{Y_1+\ldots+Y_N}{N} - Y\Bigr| \geq \frac{\eps}{4s}} \leq
2\exp\left(\!-\frac{\eps^2N}{8s^2}\right) \leq  \delta/2
\,.
$$
Therefore, with probability at least $1-\delta/2$, we have that
$s(Y_1+\ldots+Y_N)/N$ approximates $sY$ within an error of $\eps/4$. 
It remains to approximate $W$ within error $\eps/4$ and with probability at
least $1 - \delta/2$. That can be done
easily by sampling $O(\eps^{-2}\log\delta^{-1}\log^2n)$ independent
uniform random vertices of $G$ and taking the average of the logarithm of
twice their degrees, $\widetilde{W}$ (see the last step %
of
Algorithm \ref{alg:spanningtree}). Since $\log \bigl(2\weight(y)\bigr)\leq 2\log n$ for
all $y\in V$, again by Hoeffding's inequality we have
$$ \P{|\widetilde{W} - W| \geq \frac{\eps}{4}} \leq \delta/2\,.$$
Therefore, by the union bound the algorithm succeeds with probability at
least $1-\delta$.

It remains to compute the number of oracle accesses. We used
$O({\eps^{-2}}{\log\delta^{-1}\log^2n})$ accesses to approximate $W$. 
On the other
hand, we can compute each $Y_i$ with at most $2r = O(\eps^{-3})$ oracle
accesses. Therefore, we can approximate $Y$ with at most
$$ 2Nr = O\bigl({\eps^{-5}}{\log \delta^{-1}s^2}\big)
= O\bigl({\eps^{-5}}{\log^2\eps^{-1}\log\delta^{-1}}\big)
= \tilde{O}\bigl(\eps^{-5}\log \delta^{-1}\big)$$
many queries. 
\end{proof}

We remark that knowing $|E|$ is not really necessary for this algorithm,
since it contributes a term of size $O(n^{-1}\log n)$, which will be much
less than $\epsilon$ in any reasonable example where one might use this
algorithm.

\section{Bounds for Vertex-Transitive Graphs}
\label{sec:transitive}
Let $G$ be a weighted, locally finite, vertex-transitive
graph. We recall that $G$ is \dfn{vertex transitive} if for every two vertices
$x,y\in V$, there is an automorphism $\phi \colon  G\rightarrow G$ such
that $\phi(x)=y$. 
Since $G$ is transitive, it is a $\weight$-regular graph, where
$\weight=\weight(x)$ for every $x\in V$. 

For a vertex $x\in G$ and $r\geq 0$, let $N(x,r):=|B_{\dist}(x,r)|$. Note
that we are using the cardinality here, not the volume. Since
$G$ is vertex transitive, $N(x,r)=N(y,r)$ for every two vertices $x,y\in V$. Therefore, we may drop the index $x$ and use $N(r)$.

The following theorem is the main result of this section. 
The first part appears to be new.
The last part is an improvement over the known
\eqref{eq:transitivegapintro}.

\begin{theorem}
\label{thm:transitiveeig}
Let $G$ be a connected, weighted, vertex-transitive, locally finite graph
such that $\weight(x,y)\geq 1$ for all adjacent pairs of vertices.
If $\alpha\in (0,1)$,
$\>0 < \delta\le 2/\weight$,
and $x\in V$, then
\begin{equation}
\label{eq:transitiveeig}
\muz_x(\delta) = \muz(\delta) 
\leq
\frac{1}{(1-\alpha)^2 
N\Bigl(\frac{\textstyle\arcsin \sqrt{\alpha/2}}{\textstyle\arcsin
\sqrt{\weight\delta/2}}\Bigr)}
\le
\frac{1}{(1-\alpha)^2 N\big(\!\sqrt{\alpha/(\weight\delta)}\,\big)}\,.
\end{equation}
In addition, if $G$ is finite, then 
\begin{equation}
\label{eq:improvedl2}
\lambda_2
>
\frac2\weight
\left(\sin \frac{\pi}{4 \diam} \right)^2
\,.
\end{equation}
\end{theorem}

First we show that the spectral projections are equivariant with respect to
the automorphisms of the graph $G$.
Here, we identify an automorphism $\phi$ of $G$ with the unitary operator
$f \mapsto \phi f$ that it induces, where $(\phi f)(x) :=
f\bigl(\phi^{-1}(x)\big)$ for $f \in \ell^2(V, \weight)$ and $x \in V$.
Consequently, the spectral measures of vertices are the same for all
vertices.

\begin{fact}
\label{fact:commute}
Every automorphism operator $\phi$ commutes with the Laplacian, i.e., $\phi \cL = \cL \phi$. 
Since $I(\delta)$ is a function of $\cL$, every automorphism $\phi$ also
commutes with $I(\delta)$ and $\Iz(\delta)$, as does $\cL$.  
\end{fact}
\begin{proof}
The first fact is clear from the definition of the Laplacian. The second
follows from the symbolic calculus.
\end{proof}

\begin{lemma}
For every two vertices $x,y\in V$ and every $\delta \geq 0$, we have
$\muz_x(\delta)=\muz_y(\delta)$. 
\end{lemma}
\begin{proof}
Choose an automorphism $\phi$ such that $\phi \b1_y = \b1_x$. Since $\phi$
commutes with $\Iz(\delta)$, we have $\Iz(\delta) = \phi^{-1} \Iz(\delta)
\phi$. 
Therefore, 
$$ \muz_x(\delta)=\langle \Iz(\delta) \b1_x, \b1_x \rangle = \langle \Iz(\delta) \phi \b1_y, \phi \b1_y\rangle = 
\langle \phi^{-1} \Iz(\delta) \phi \b1_y, \b1_y\rangle = \langle \Iz(\delta)
\b1_y, \b1_y\rangle = \muz_y(\delta)\,, $$ 
where the third equation follows by the fact that $\phi$ is a unitary
operator. 
\end{proof}

The next two lemmas show particular properties of the spectral embedding of vertex-transitive graphs. 
The first part was observed for finite graphs in \cite[Proposition
1]{DelTil}.
\begin{lemma}
\label{lem:normequaltran}
For every two vertices $x,y\in V$, we have $\norm{F_x} = \norm{F_y}$.
Furthermore, for every automorphism $\phi$,
$$\norm{F_x-F_y} = \norm{F_{\phi(x)}-F_{\phi(y)}}\,.$$
\end{lemma}
\begin{proof}
Choose an automorphism $\phi$ such that $\phi \b1_y = \b1_x$.  Since 
$\phi \Iz(\delta) = \Iz(\delta) \phi$
by \autoref{fact:commute}, we have
$$ F_x = \Iz(\delta)\b1_x/\weight =
\Iz(\delta)\phi \b1_y/\weight = 
\phi \Iz(\delta) \b1_y/\weight = \phi F_y\,.$$
Since $\phi$ is a unitary operator, it preserves the norm, thus $\norm{F_x} = \norm{F_y}$. 
We also proved that $F_{\phi(y)} = \phi\bigl(F_y\big)$. Therefore
\[
\norm{F_{\phi(x)} - F_{\phi(y)}} 
= \norm{\phi\bigl(F_x - F_y\big)} = \norm{F_x-F_y}
\,.
\qedhere
\]
\end{proof}

\begin{lemma}
\label{lem:rayleightran}
For every weighted, vertex-transitive graph $G$, every vertex $x\in V$,  and
$\delta\in (0,2)$, 
\begin{equation}
\label{eq:energyatx}
\sum_{y\sim x} \weight(x,y)\norm{F_x-F_y}^2 
=2\langle F_x,\cL F_x\rangle_\weight
\,.
\end{equation}
Hence,
if $\weight(x,y)\geq 1$ for all $y\sim x$, then 
$$ \delta \geq  \frac{\max_{y\sim x}
\norm{F_x-F_y}^2}{2\norm{F_x}_\weight^2}\,.$$
\end{lemma}

\begin{proof}
 Fix $x \in V$.
 By \autoref{lem:normequaltran}, we have
\begin{align*} 
\sum_{y\sim x} \weight(x,y)\norm{F_x-F_y}^2 
&=
2\weight \norm{F_x}^2 - 2\sum_{y\sim x} \weight(x,y)\langle F_x,F_y\rangle\\
&=
2 \Bigl\langle F_x, \sum_{y\sim x} \weight(x,y)\bigl(F_x - F_y\big)\Bigr\rangle.
\end{align*}
On the other hand, the definition of $\cL$ yields
\[
\weight\, \cL \b1_x(z)
=
\begin{cases}
\weight &\text{if } z = x,\\
-\weight(x, z) &\text{if } z \sim x,\\
\end{cases}
\]
which is to say that
\[
\weight\, \cL \b1_x 
= \sum_{y\sim x}\weight(x,y) (\b1_x-\b1_y)
\,.
\]
Therefore,
\begin{align*} \sum_{y\sim x} \weight(x,y)\bigl(F_x - F_y\big)
&= \frac{1}{\weight}\sum_{y\sim x} \weight(x,y)(\Iz(\delta)\b1_x - \Iz(\delta)\b1_y) 
\\ &=
\frac{\Iz(\delta)}{\weight}\sum_{y\sim x}\weight(x,y) (\b1_x-\b1_y)
= \Iz(\delta) \cL \b1_x = \cL \Iz(\delta) \b1_x = \weight\,\cL F_x\,.
\end{align*}
Putting together the above equations, we obtain \eqref{eq:energyatx}.
Therefore,
$$ \max_{y\sim x} \norm{F_x - F_y}^2 \leq  2\langle F_x, \cL
F_x\rangle_\weight   \leq 2\delta \langle F_x,F_x\rangle_\weight\,,$$
where the first inequality holds by the assumption that $w(x,y)\geq 1$
and the second inequality holds by
\autoref{lem:rayleighquotient} and the fact that $F_x=\Iz(\delta)\b1_x/w \in
\img\bigl(I(\delta)\big)$. 
\end{proof}

\begin{proofof}{\autoref{thm:transitiveeig}}
The second inequality of \eqref{eq:transitiveeig} is a consequence of the
fact that $(\sin x )/ x$ is decreasing for $0 < x < \pi$.

For a vertex $x\in V$, let $\beta(x):=\max_{y\sim x}
\norm{F_x-F_y}_\weight^2$.
Since $G$ is vertex transitive, \autoref{lem:normequaltran} tells us that
$\beta(x) = \beta(y)$ for every two vertices $x,y\in V$. Therefore, we may
drop the $x$ and write just $\beta$. Likewise, we may write $\rho := \wnorm{F_x}$
for the radius of the sphere (about $\bf 0$) 
in $\ell^2(V, \weight)$ that contains all points
$F_y$ ($y \in V$). By \autoref{lem:rayleightran},
\begin{equation*}
\delta 
\geq  \frac{\max_{y\sim x} \norm{F_x-F_y}_\weight^2}
{2\weight\norm{F_x}_\weight^2} 
= \frac{\beta}{2\weight\rho^2}\,.
\end{equation*}

Given a pair $x, y\in V$ of vertices, write $\theta(x, y)$ for the angle
between $F_x$ and $F_y$. 
For $x \sim y$, the distance on the sphere of radius $\rho$ in $\ell^2(V,
\weight)$
between $F_x$ and $F_y$ equals $\rho \cdot \theta(x, y)$, where $\rho \sin
\bigl(\theta(x, y)/2\bigr) =
\wnorm{F_x - F_y}/2 \le \sqrt\beta/2$. Thus,
\[
x \sim y \quad \Longrightarrow \quad
\theta(x, y) \le 
2 \arcsin \frac{\sqrt\beta}{2\rho}
\le 2 \arcsin \sqrt{\weight \delta/2}
\,.
\]
Therefore, every pair $x, y \in V$ satisfies
$\theta(x,y) \le 2\dist(x, y) \arcsin \sqrt{\weight \delta/2}$ by the
triangle inequality for the sphere metric.
Also,
$\rho^2 \cos \theta(x, y) =
\wiprod{F_x,F_y} = F_x(y)$.
It follows that $F_x(y) \ge (1 - \alpha) F_x(x) = (1- \alpha) \rho^2$ when 
\[
\dist(x, y) \le 
\frac{\arccos (1 - \alpha)}{2 \arcsin \sqrt{\weight \delta/2}}
=
\frac{\arcsin \sqrt{\alpha/2}}{\arcsin \sqrt{\weight \delta/2}}
=: r
\,.
\]
Since 
\[
\rho^2
=
\wnorm{F_x}^2
=
\sum_{y \in V} F_x(y)^2 \weight
\ge
\sum_{y \in B(x, r)} F_x(y)^2 \weight
\ge
N(r) (1 - \alpha)^2 \rho^4 \weight
\,,
\]
it follows from \autoref{lem:normF} that
\[
\muz(\delta)
=
\rho^2 \weight
\le
\frac1{(1 - \alpha)^2 N(r)}
\,,
\]
as desired for \eqref{eq:transitiveeig}.

Now choose $\delta := \lambda_2$.
Since \eqref{eq:improvedl2} is trivial when $\lambda_2 \weight/2 > 1$,
we may assume that  $\lambda_2 \weight/2 \le 1$.
Let $R$ be the shortest distance from $x$ to a vertex 
$y$ where $F_x(y) < 0$. Such a vertex $y$ exists because $F$
is \spheresym\ by \autoref{lem:spheresym}.
Then as above, it follows that
\[
\pi/2
<
\theta(x, y)
\le
2 R \arcsin \sqrt{\lambda_2 \weight/2}
\,.
\]
Therefore,
\[
\sqrt{\frac{\lambda_2 \weight}2}
>
\sin \frac\pi{4R}
\ge
\sin \frac\pi{4\diam}
\,,
\]
from which \eqref{eq:improvedl2} follows.
\end{proofof}

\autoref{thm:transitiveeig} is strong enough to give known sharp results.
We give two examples, with proofs relegated to the appendix. 

In the first corollary, we give an upper bound on the return probability
of the random walks for graphs with polynomial (or faster) growth rate.

\begin{corollary} \label{cor:polygrowthtrans}
Let $G$ be a finite or 
infinite, unweighted, loopless, $d$-regular, vertex-transitive graph with
at least polynomial growth rate $N(r)\geq Cr^D$, where $C > 0$ and $D \ge 1$ are
constants and $0 \le r \le \diam$.
Then for every $x\in V$, every $\delta \in
(0, 2)$, and every $t > 0$,
\[
\muz_x(\delta) 
\le
C'\delta^{D/2}
\,,
\]
where
\[
C' := \frac{4 D^2 d^{D/2}}{C}
\,,
\]
and lazy simple random walk satisfies
$$ p_t(x,x) - \pi(x) \leq
C'' t^{-D/2}
\,,
$$
where
$$
C'' :=
\frac{8 D^{(D+5)/2} d^{D/2}}{C e^{D/2}}
\,.
$$
\end{corollary}

Without an explicit constant, such a result concerning return probabilities
was first proved for infinite Cayley
graphs in the celebrated
breakthroughs of Varopoulos; see \cite[Corollary 7.3]{CGP}.
The case of finite Cayley graphs was done in a similar result of
Diaconis and Saloff-Coste \cite[Theorem 2.3]{DSC:modgr}.
For results on uniform mixing time, see \cite{GMT}.

For comparison, the usual Cayley graph of $\Z^D$ has 
\[
\muz_x(\delta) =
|\{(s_1, \ldots, s_D) \in (\R/\Z)^D : 1 - \sum_{i=1}^D \cos(2\pi s_i)/D \le \delta\}|
\,.
\]
We may identify $(\R/\Z)^D$ with $[-1/2, 1/2)^D$.
Since $1 - \sum_{i=1}^D \cos(2\pi s_i)/D \le 2\pi^2 \sum_i s_i^2/D$ and the
\vadjust{\kern2pt}%
unit ball in $\R^D$ 
\vadjust{\kern1pt}%
has volume $\pi^{D/2}/\Gamma(D/2+1)$, it follows that
\[
\muz_x(\delta) \ge \bigl(D/(2\pi)\big)^{D/2}\Gamma(D/2+1)^{-1} \delta^{D/2}
\]
(and is asymptotic to this as $\delta\to 0$).

This can also be compared to the general results known for return
probabilities using isoperimetric information. For example, the method of
evolving sets \cite{MorPer} gives that
lazy simple random walk satisfies
\[
p_t(x,x) - \pi(x) \leq
\frac{8 (8D)^{D/2} d^{D}}{C} (t-1)^{-D/2}
\,,
\]
whose dependence on both $D$ and $d$ is worse than ours. 
Similarly, the bounds on both spectral measure and on return probabilities
(for continuous-time random walk) of \cite{ChYau} for finite graphs are
worse than ours.
In the infinite case, simple random walk that is not necessarily lazy was
shown by \cite{SC95} to satisfy 
\[
p_{2t}(x, x) \le
\frac{(2+D)^{1+D/2}2^{D/2}}{C} (2t)^{-D/2}
\]
when $N(r) \ge C(r+1)^D$. 
Our bound is better when $d \le 5$ and $D$ is large.
Although $d \ge 2D$ for Cayley graphs of abelian groups,
there are nilpotent groups of arbitrarily large (polynomial)
growth rate that can be generated by only two elements.

The next corollary gives comparable results for super-polynomial growth.

\begin{corollary}
\label{cor:superpolytransitive}
For every finite or
infinite, unweighted, loopless, $d$-regular, vertex-transitive graph $G$
with super-polynomial
growth rate $N(r)\geq c_1 \exp\{c_2 r^a\}$ for some $0 < a
\le 1$ and positive constants $c_1, c_2$ ($0 \le r \le \diam$),
and every $x\in V$,
$$
\muz_x(\delta)
\le
\frac{4}{c_1} \exp \{-c_3 \delta^{-a/2}\}
\,,
$$
where 
$$
c_3 := c_2 (2d)^{-a/2}
\,,
$$
and lazy simple random walk satisfies
$$ 
p_t(x,x) - \pi(x) \leq \frac{4}{c_1} \bigl(1 + c_4
t^{\frac{a}{a+2}}\big) \exp\{-c_4 t^{\frac{a}{a+2}}\}\,,$$
where 
$$
c_4 :=\left(\frac{c_2^{2/a}}{2ad}\right)^{\frac{a}{a+2}} 
.
$$
\end{corollary}

Without an explicit constant, such a result concerning return probabilities
was first proved for infinite Cayley
graphs by Varopoulos; see \cite[Corollary 7.4]{CGP}.
The case of finite Cayley graphs was done in a similar result of
\cite[Theorem 2.7]{DSC:modgr}.
A similar bound was proved by \cite{SC95}: if $N(r) \ge \exp\{c_2 r^a\}$
for large $r$, then simple random walk (without assuming laziness)
satisfies
\[
p_{2t}(x, x) = O\bigl(\exp\{-c_5 (2t)^{a/(a+2)}\}\bigr)
\]
for
\[
c_5 < \frac{e^{2/a-1} (a c_2)^{2/a}}{4^{(a+1)/(a+2)}}
\,.
\]
Whether this bound is worse or better than ours depends on the constants
involved and the degree.

Our bound \autoref{cor:superpolytransitive}
on the spectral measure is also proved (without explicit
constants) by \cite[Corollary 1.8]{BPS} for infinite amenable Cayley
graphs.
We remark that a better bound on the return probabilities
can be obtained by choosing $\alpha$ in the
proof closer to 1 or even as a function of $\lambda$.

For additional information on the spectrum of infinite groups, see
\cite{BPS,BBP}.

\section{Open Questions} \label{sec:open}

Can the method of spectral embedding be used to give new proofs of bounds
on spectral measure or return probabilities under hypotheses involving
isoperimetric (expansion) profiles? If so, will sharper bounds result?

An open question \cite[e.g., after Cor.~5.3]{KomPer} is whether for every
finite, transitive, unweighted graph, the $L^1$-mixing time is
$O(\weight \diam^2)$. Can spectral embedding help in answering this?

The Grone--Merris conjecture \cite{GM94}
gives lower bounds for eigenvalues in terms of
the degrees of a finite, unweighted, loopless graph.
This concerns the eigenvalues $0 = \widetilde\lambda_1 <
\widetilde\lambda_2 \le \widetilde\lambda_3 \le \cdots \le \widetilde\lambda_n$
of the unnormalized combinatorial Laplacian, $\Delta$, i.e., the degree
matrix minus the adjacency matrix.
Namely, if $d'_k$ denotes the number of vertices whose degree is at least
$k$, then $\sum_{i=1}^k \widetilde\lambda_i \ge \sum_{i=1}^k d'_{n-i+1}$
for every $k \in [0, n-1]$.
The conjecture was proved by \cite{Bai}. Using
spectral embedding, we may formulate this result as a purely geometric
statement in Euclidean space: every
orthogonal projection $F \colon \ell^2(V) \to \ell^2(V)$
of rank $k$ has energy $\sum_{x \sim y} \norm{F_x - F_y}^2 \ge \sum_{v \in V}
\bigl(\weight(v) - (n-k)\bigr)^+$, where $F_x := F(\b1_x)$. 
To see that this is an equivalent formulation, let
$\langle \varphi_i \rangle$ be orthonormal with $\Delta \varphi_i =
\widetilde \lambda_i \varphi_i$. 
Then $\sum_{x \sim y} \norm{F_x - F_y}^2 = \sum_i
\widetilde\lambda_i a_i$, where $a_i := \norm{F(\varphi_i)}^2$.
We have $0 \le a_i \le 1$ and $\sum_i a_i =
k$, whence the minimum occurs for $F$ equal to the projection on the span of
$\varphi_1, \ldots, \varphi_k$.
Does this formulation lead to a simpler proof? Can spectral embedding be
used to establish a generalization to simplicial complexes \cite{DuvRei}?

\subsubsection*{Acknowledgements.} We are grateful to Yuval Peres for several
useful discussions. We also thank Laurent Saloff-Coste and Prasad Tetali for helpful comments on earlier versions of this document. In addition, prior unpublished
work of Lyons with Fedja Nazarov
obtained via different methods
an upper bound of $2.84/t^{1/4}$ for \autoref{thm:universalavg} and
$10/r^{1/4}$ for \autoref{cor:counting}.
We thank two referees for catching numerous small errors.

\appendix
\section{Appendix: Miscellaneous Proofs}

The first proposition is based on \cite[(16)--(19)]{MorPer}.

\begin{proposition}
\label{prop:L2returnprob}
For lazy random walk on every finite graph $G$ and $\eps>0$, we have
$$
\tau_2(\eps) = \min\left\{t: \forall x\in V\
\frac{p_{2t}(x,x)}{\pi(x)}\leq 1+\eps^2
\right\}
=
\ceil{\tau_\infty(\epsilon^2)/2}
\,.
$$
\end{proposition}
\begin{proof}
For every $t>0$ and vertex $x\in V$, we have
$$ \sum_{y\in V} \left(\frac{p_t(x,y)}{\pi(y)} -1\right)^2\pi(y) =
\sum_{y\in V} \frac{p^2_t(x,y)}{\pi(y)} - 1 = \sum_{y\in V}
\frac{p_t(x,y)p_t(y,x)}{\pi(x)}-1 = \frac{p_{2t}(x,x)}{\pi(x)}-1\,,$$
where in the first equality we used the fact that $p_t(x,\cdot)$ and
$\pi(\cdot)$ are probability distributions, and  the second equality follows by the fact that $p_t(x,y)\pi(x)=p_t(y,x)\pi(y)$.

In addition, stationarity of $\pi$ gives
\begin{align*}
\left|\frac{p_{2t}(x, y) - \pi(y)}{\pi(y)}\right|
&=
\left|\frac{1}{\pi(y)}
\sum_z \bigl(p_t(x, z) - \pi(z)\big)(p_t(z, y) - \pi(y)\big)\right|
\\ &=
\left|
\sum_z \pi(z) \Big(\frac{p_t(x, z)}{\pi(z)} - 1\Big)\Big(\frac{p_t(z,
y)}{\pi(y)} - 1\Big)\right|
\\ &=
\left|
\sum_z \pi(z) \Big(\frac{p_t(x, z)}{\pi(z)} - 1\Big)\Big(\frac{p_t(y,
z)}{\pi(z)} - 1\Big)\right|
\\ &\le
\sqrt{
\sum_z \pi(z) \Big(\frac{p_t(x, z)}{\pi(z)} - 1\Big)^2}
\sqrt{
\sum_z \pi(z)\Big(\frac{p_t(y,z)}{\pi(z)} - 1\Big)^2}
\\ &=
\sqrt{\frac{p_{2t}(x,x)}{\pi(x)}-1} \sqrt{\frac{p_{2t}(y,y)}{\pi(y)}-1}
\,.
\end{align*}

Since $t \mapsto p_t(x, x)$ is monotone decreasing by laziness,
the result follows.
\end{proof}

\begin{proofof}{\autoref{lem:reversereturnprobspectral}} 
We have
\[
p_t(x, x) - \pi(x)
=
\int_0^1 (1-\lambda)^t \,d\muz_x(\lambda)
\ge
(1-\delta)^t \muz_x(\delta)
\,,
\]
whence
\[
\muz_x(\delta)
\le
(1-\delta)^{-t} \bigl(p_t(x, x) -\pi(x)\big)
\,.
\]
Since $\log (1-s) > -s/(1-s)$ for $0 < s < 1$, it follows that for 
$0 < \delta \le 1/2$ and
$t := \flr{ 1/\delta}$,
\[
(1-\delta)^{-t} 
<
(1 - \delta)^{-1} \exp \left(\frac{(t-1)\delta}{1-\delta}\right)
\le
(1 - \delta)^{-1} e
\le
2 e \,.
\]
This proves the first inequality.

The second inequality is a little simpler: 
\[
q_t(x, x) - \pi(x)
=
\int_0^2 e^{-\lambda t} \,d\muz_x(\lambda)
\ge
e^{-\delta t} \muz_x(\delta)
\,.
\]
Substitution of $t := 1/\delta$ gives the result.
\end{proofof} 

\begin{proofof}{\autoref{lem:bothreturnprobspectral}}
The proofs of the two parts are essentially the same, so we give only the
first.
We have
\begin{align*}
p_t(x, x) - \pi(x)
&=
\int_0^1 (1-\lambda)^t \,d\muz_x(\lambda)
=
\int_{(0,\delta]} (1-\lambda)^t \,d\muz_x(\lambda)
+
\int_{(\delta, 1]} (1-\lambda)^t \,d\muz_x(\lambda)
\\ &\le
\muz_x(\delta) + (1-\delta)^{\flr{t/2}} 
\int_{(\delta, 1]} (1-\lambda)^{\ceil{t/2}} \,d\muz_x(\lambda)
\\ &\le
\muz_x(\delta) + (1-\delta)^{\flr{t/2}} 
\int_{(0, 1]} (1-\lambda)^{\ceil{t/2}} \,d\muz_x(\lambda)
\\ &=
\muz_x(\delta) + (1 - \delta)^{\flr{t/2}} \bigl(p_{\ceil{t/2}}(x, x) -
\pi(x)\big)
\,.
\end{align*}
Rearranging gives the result.
\end{proofof} 

\begin{proofofnoqed}{\autoref{thm:polyvertexmeasure}}
Define
$$
r_0 :=
\left(\frac{\weight(x)}{c(1-\alpha)^2\muz_x(\delta)}\right)^{1/D}
\,.
$$
Since $\weight(x) = \vol(x, 0) \ge c$, we have $r_0 > 1$.
Take $r := \ceil{r_0} - 1$. Then
the hypothesis $\vol(x,r) >
\frac{\weight(x)}{(1-\alpha)^2\muz_x(\delta)}$ of
\autoref{thm:generalvertexmeasure} is satisfied and
$
r \le r_0$.
Substitution of this bound in
\autoref{eq:generalvertexmeasure} 
with the choice $\alpha := D/(D+1)$ gives the
claimed upper bound on $\muz_x(\delta)$.

For lazy simple random walk, we may take $G$ to be loopless and $\cL_{\rm
lazy} = \cL/2$.
Now use \autoref{lem:returnprobspectral} to get that 
\begin{align*}
p_t(x, x) 
&\le \int_0^1 e^{-\lambda t} C w(x) \frac{2^{D/(D+1)} D}{D+1}
\lambda^{-1/(D+1)} \,d\lambda
\\ &< 
\frac{C w(x) D}{D+1} \left(\frac{2}{t}\right)^{D/(D+1)} \int_0^\infty e^{-s}
s^{-1/(D+1)} \,ds
=
C' w(x) t^{-D/(D+1)}
\,.
 \tag*{\qed}
\end{align*}
\end{proofofnoqed}

\begin{proofof}{\autoref{prop:expanddet}}
Write $\detp A$ for the product of the non-zero eigenvalues of a matrix $A$.
As shown by \cite{RungeSachs}, we may rewrite the matrix-tree theorem as
$$
\cp(\gh) = \frac{\prod_{x \in \verts} 2\weight(x)}{\sum_{x \in
\verts} 2\weight(x)} \detp (I-P)
$$
[the proof follows from looking at the coefficient of $s$ in
$\det\bigl(I-P-s I\big) = (\det 2D)^{-1} \det(\Delta - 2s
D)$ and using the matrix-tree theorem in its original form with
cofactors, where $D$ is the diagonal degree matrix and $\Delta := 2D(I - P)$].
Thus,
\begin{equation}
\label{eq:mtt'}
\log \cp(\gh)
= 
-\log \bigl(4|\edges|\big) + \sum_{x \in \verts} \log
\bigl(2\weight(x)\bigr)
+ \log \detp (I-P)
\,.
\end{equation}
Let $\hat\lambda_k$ be the eigenvalues of $P$ with $\hat\lambda_1 = 1$.
We may rewrite the last term of \eqref {eq:mtt'} as 
\begin{align*}
\log \detp (I-P)
&=
\sum_{k=2}^n \log (1 - \hat\lambda_k)
=
- \sum_{k=2}^n \sum_{t \ge 1} \hat\lambda_k^t/t
\\ &=
- \sum_{t \ge 1} \sum_{k=2}^n \hat\lambda_k^t/t
=
- \sum_{t \ge 1} \frac{1}{t} (\tr P^t - 1)
\,.
\end{align*}
Since $\tr P^t = \sum_{x \in \verts} p_t(x, x)$, the desired formula
now follows from this and \eqref{eq:mtt'}.
\end{proofof}

\begin{proofof}{\autoref{cor:polygrowthtrans}}
First, note that for $\sqrt{\alpha/(d\delta)} \le \diam$, we have
\begin{equation*}
\muz(\delta) \leq
\frac{1}{(1-\alpha)^2 N\bigl(\sqrt{\alpha/(d\delta)}\big)}
\le
\frac{(d\delta)^{D/2}}{C(1-\alpha)^2 \alpha^{D/2}}
\end{equation*}
by \eqref{eq:transitiveeig}.
In particular, this holds for $\delta \ge 1/(d\diam^2)$, whence for
$\delta \ge \lambda_2$ in the finite case.
Since $\muz(\delta) = 0$ for $\delta < \lambda_2$, it follows that the bound
\begin{equation}
\label{eq:transitivepolygrowth}
\muz(\delta) \leq
\frac{(d\delta)^{D/2}}{C(1-\alpha)^2 \alpha^{D/2}}
\end{equation}
applies for all $\delta > 0$ even when $G$ is finite.

Now, set $\alpha:=D/(D+4)$. 
With 
$C_0' := \frac{(D+4)^{D/2+2}d^{D/2}}{16 C D^{D/2}} \le C'$,
a sharper version of the first inequality, $\muz_x(\delta) \le C_0'
\delta^{D/2}$, is immediate from \eqref{eq:transitivepolygrowth}.
Therefore, \autoref{lem:returnprobspectral} allows us to write
\begin{align*}
p_t(x,x) - \pi(x) 
&\leq \int_0^1 e^{-\lambda t}
\left(\frac{(D+4)^{D/2+2}}{16 C D^{D/2}}(d\cdot2\lambda)^{D/2} \right)'\, d\lambda  
\\ &=
 \frac{(D+4)^{D/2+2} (2d)^{D/2}}{32 C D^{D/2-1}}\int_0^1 e^{-\lambda t}
 \lambda^{D/2-1}\, d\lambda\\
&<  \frac{(D+4)^{D/2+2}}{32 C D^{D/2-1}}\left(\frac{2d}{t}\right)^{D/2}
\int_0^\infty e^{-s} s^{D/2-1}\, ds\\
&= \frac{(D+4)^{D/2+2}}{32 C D^{D/2-1}}
\Gamma\left(\frac{D}{2}\right)\left(\frac{2d}{t}\right)^{D/2}\,.
\end{align*}
This gives the second inequality with the better constant $C_0'' := 
\frac{(D+4)^{D/2+2}(2d)^{D/2}}{32 C D^{D/2-1}}
\Gamma\left(\frac{D}{2}\right) \le C''$.
\end{proofof}

\begin{proofofnoqed}{\autoref{cor:superpolytransitive}}
As in the proof of \autoref{cor:polygrowthtrans}, we may ignore the
restriction on $r$ when substituting the growth condition into 
\eqref{eq:transitiveeig}.

The bound on $\muz_x(\delta)$ is immediate from
\autoref{thm:transitiveeig} with the choice $\alpha := 1/2$.

Define $\beta(t):=c_4 t^{\frac{a}{a+2}}$.
By \autoref{lem:returnprobspectral}, we can then write
\begin{align*}
p_t(x,x) - \pi(x) 
&\leq \int_0^1 e^{-\lambda t} \left( \frac{4}{c_1
\exp(c_2(2d\cdot2\lambda)^{-a/2})}\right)' \,d\lambda
\\ &= \frac{2a c_2 (4d)^{-a/2}}{c_1}\int_0^1 e^{-\lambda /
- c_2 (4d\lambda)^{-a/2}}\lambda^{-\frac{a+2}{2}} \,d\lambda
\,.
\end{align*}
Since $\lambda \mapsto
-\lambda t - c_2 (4d\lambda)^{-a/2} $ is a concave function, it is
maximized at the point
$\lambda_* :=\left(\frac{2t}{c_2 a}\right)^{-\frac{2}{a+2}}
(4d)^{-\frac{a}{a+2}}$. Therefore
$$ -\lambda t - c_2 (4d\lambda)^{-a/2} \leq
- c_2 (4d\lambda_*)^{-a/2} 
= -\beta(t)\,.$$
Therefore, 
\begin{align*}
 p_t(x,x) - \pi(x) &\leq  \int_0^{\lambda_*} \left( \frac{4}{c_1
 \exp(c_2(4d\lambda)^{-a/2})}\right)' \,d\lambda+ \frac{1}{c_1}
 2ac_2(4d)^{-a/2}e^{-\beta(t)}
 \int_{\lambda_*}^\infty\lambda^{-\frac{a+2}{2}}\,d\lambda 
 \\ &=  \frac{4}{c_1}\left( e^{- \beta(t)}
 + \beta(t)e^{-\beta(t)}\right) \,.
 \tag*{\qed}
 \end{align*}
\end{proofofnoqed}

\addcontentsline{toc}{section}{References}

\bibliographystyle{alpha}
\bibliography{references}

\end{document}